\documentclass{article}
\pdfoutput=1
\usepackage[margin=1.25in]{geometry}
\usepackage{authblk}
\usepackage{graphicx} 
\usepackage{amsmath}
\usepackage{amsthm}
\usepackage{amsfonts}
\usepackage{amssymb}
\usepackage{xcolor}
\usepackage{aligned-overset}
\usepackage{mathtools}
\usepackage{url}
\usepackage{algorithm}
\usepackage{algorithmic}

\usepackage{tikz}
\usetikzlibrary{decorations.pathmorphing}
\usepackage{pgfplots}
\pgfplotsset{compat=1.18}
\usetikzlibrary{fillbetween}
\usetikzlibrary{intersections}
\usetikzlibrary{patterns}

\newtheorem{theorem}{Theorem}[section]

\newtheorem{lemma}[theorem]{Lemma}
\newtheorem{fact}[theorem]{Fact}
\newtheorem{definition}{Definition}

\newcommand{\vol}{\mathrm{vol}}

\newcommand{\bulkset}{C_2}
\newcommand{\bndryset}{C_1}

\newcommand{\bulksetbody}{\mathcal{C}_2}
\newcommand{\bndrysetbody}{\mathcal{C}_1}

\newcommand{\facet}{\beta}

\newcommand{\curlyext}{\mathcal{E}_u(\mathcal{P}_{\facet})}

\DeclareMathOperator*{\argmax}{arg\,max}
\DeclareMathOperator*{\argmin}{arg\,min}

\definecolor{darkgreen}{rgb}{0,0.6,0}
\definecolor{purple}{rgb}{0.5,0,0.5}

\title{
A mixing time bound for Gibbs sampling from log-smooth log-concave distributions}
\makeatletter

\author{Neha S. Wadia\thanks{neha.wadia@berkeley.edu}}

\affil{\small Center for Computational Mathematics, Flatiron Institute, New York NY 10010, USA}

\begin{document}

\maketitle

\begin{abstract}
The Gibbs sampler, also known as the coordinate hit-and-run algorithm, 
is a Markov chain that is widely used to draw samples
from probability distributions in arbitrary dimensions.
At each iteration of the algorithm, a randomly selected coordinate is resampled from the
distribution that results from conditioning on all the other coordinates.
We study the behavior of the Gibbs sampler on the class of log-smooth 
and strongly log-concave target distributions 
supported on $\mathbb{R}^n$.
Assuming the initial distribution is $M$-warm with respect to the target, we show that the Gibbs sampler
requires at most
$O^{\star}\left(\kappa^2 n^{7.5}\left(\max\left\{1,\sqrt{\frac{1}{n}\log \frac{2M}{\gamma}}\right\}\right)^2\right)$
steps to produce a sample with error no more than $\gamma$ in total variation distance
from a distribution with condition number $\kappa$.
\end{abstract}

\section{Introduction}

Sampling from probability distributions in high dimensional spaces is a 
fundamental computational primitive;
it forms the basis of efficient numerical methods for approximating arbitrary integrals.
The problem statement is the following: given a density function $\pi$, compute a point $x$ with density proportional to $\pi(x)$.

A general approach to solving this problem is to design a reversible, ergodic Markov 
chain with a unique stationary distribution that is
equal to the target distribution from which samples are needed. 
It is often possible to design relatively simple chains with low per-iteration computational
complexity that are fit for purpose by
implementing the Metropolis-Hastings filter \cite{metropolis1953-eq-of-state,hastings1970}, a rule by which to either accept
the next step in the dynamics or remain put and so tailor the dynamics
toward a specific stationary distribution.
The resulting \emph{Metropolized} or \emph{Markov Chain Monte Carlo}
algorithms are known to 
converge asymptotically to their stationary distributions under mild 
regularity conditions.
Non-asymptotic rates of convergence or \emph{mixing times}
are comparatively few in number and are both algorithm- and target-specific. They are important because downstream estimators computed using samples drawn from a dynamics that has not converged will suffer from bias.

The class of log-concave target distributions is of particular interest.
These are distributions with density functions of the form $e^{-f(x)}, \, x\in\mathbb{R}^n,$ 
where $f$ is convex in $x$.
The mixing times of a number of Markov chains are polynomial in $n$ for this 
class of target distributions.
Such chains
are considered to be \emph{rapidly mixing}.
Examples are the ball walk \cite{KLS1997-random-walks-volume-algorithm,vempala-lovasz-2007}, the grid walk \cite{dyer-frieze-kannan1991,lovasz-vempala-simulated-annealing}, and 
the hit-and-run walk \cite{boneh1983-hit-and-run,smith1984-hit-and-run,lovasz-vempala2004-hit-and-run-corner}. An expository survey of the associated results may be found in \cite{vempala-review}.
In this paper, we upper bound the mixing time of the
Gibbs sampler, also known as coordinate hit-and-run (CHAR), for sampling from the class of log-smooth and strongly log-concave distributions supported on $\mathbb{R}^n$. That is, in addition to convexity we require $f$ to be such that the spectrum of its Hessian is uniformly bounded both above and below.
We show that the Gibbs sampler also mixes rapidly on this class of distributions.

At each iteration of the dynamics of the Gibbs sampler, a single coordinate 
$x_i$ of $x$ is 
chosen uniformly at random and replaced with a sample from its conditional 
distribution, i.e., the one-dimensional distribution over $x_i$ that 
results from conditioning on all the other $n-1$ coordinates.
This dynamics can be interpreted as a limiting case of a Metropolized algorithm in which every move is accepted.
An advantage of the Gibbs sampler is that it reduces the $n$-dimensional 
sampling problem to a one-dimensional problem at each iteration.
Another advantage is that it
has no hyperparameters to tune; in particular, there is no 
notion of a step size in the dynamics.

The Gibbs sampler has been in practice long enough to have been reintroduced
to the literature several times and is difficult to trace accurately. 
It was introduced in the theoretical computer science
literature by Turchin in 1971 \cite{turchin1971}, 
in the image processing literature by Geman and Geman in 1984 \cite{geman-geman},
and popularized in the statistics literature by Gelfand and Smith,
who demonstrated that it could be implemented efficiently on a wide variety
of problems of interest in statistics at the time, in 1990 \cite{Gelfand1990SamplingBasedAT}. 
Long before this, it was being used in statistical physics to sample
from models of magnetic materials.

An important proof technique for establishing mixing time bounds relies 
on the following two ingredients:
a result that gives the dependence of the mixing time on a quantity called
the \emph{conductance}, 
which measures how well or poorly the dynamics explores 
the state space, and an isoperimetric inequality 
that enables a lower bound on the conductance. Isoperimetric inequalities that are 
useful in this context are dynamics-dependent. 
The fact that the Gibbs sampler is constrained to move only along the $n$ 
coordinate directions prevents the use of a general technique (the localization lemma \cite{kls-isoperimetric-problems-and-localization-lemma}) for proving 
the relevant isoperimetric inequality.
Until recently, this was a barrier to establishing mixing times for target distributions supported on continuous state spaces.
(The case in discrete state spaces, where
the Gibbs sampler is often termed ``Glauber dynamics'', is different,
\cite{Martinelli1999-lecture-notes-glauber-dynamics}.)

Recently, Laddha and Vempala \cite{laddha-vempala}
and Narayanan and Srivastava \cite{narayanan2022-CHAR} concurrently 
gave bounds on the mixing time of the Gibbs sampler for sampling
from a uniform distribution supported on a convex body $K$ in $\mathbb{R}^n$
from a warm start.
Laddha and Vempala prove an isoperimetric theorem using a first-principles approach.
They obtain a bound of $O^{\star}\left(R^2n^8\right)$ on the mixing time where $R^2$ is the expected squared
distance of a random point from the centroid of $K$. 
The notation $O^{\star}$ suppresses dependence on logarithmic factors and other problem parameters.
Narayanan and Srivastava instead establish a mixing time bound for an 
auxiliary dynamics that they show dominates the mixing time of the Gibbs sampler. 
They give a bound of $O^{\star}\left(R_1^4\,n^7\right)$, where $R_1$ is 
the smallest number such that $K$ is sandwiched
between an $L_{\infty}$ ball and an $L_{\infty}$ ball scaled by $R_1$.
Narayanan, Rajaraman, and Srivastava \cite{narayananCHAR-cold-start} have since also 
demonstrated that the Gibbs sampler mixes in polynomial time for the same class of target distributions from a cold start.

We give an upper bound of $O^{\star}\left(\kappa^2 \,n^{7.5}\right)$
on the mixing time for Gibbs sampling from log-smooth and strongly log-concave distributions supported on $\mathbb{R}^n$ from a warm start.
$\kappa$ is the ratio of the upper and lower bounds on the spectrum of the Hessian of $f$.
The following theorem is our main contribution:

\begin{theorem}
    \label{thm:main-mixing-time}
    Consider a probability distribution supported on $\mathbb{R}^n$ with density 
    function proportional to $e^{-f(x)}$ where $f$ is $\mu$-strongly convex in $x$ and has $L$-Lipschitz gradients.
    Let $\kappa=L/\mu$.
    Let $\pi^{k}$ be the distribution of the $k^{th}$
    iterate produced by the Gibbs sampler.
    For 
    some universal constant $C$ and any $\gamma\in(0,1/2)$, 
    if the initial distribution $\pi^{0}$ is $M$-warm with respect to $\pi$, then the total variation distance between
    $\pi^{\tau}$ and $\pi$ is no more than $\gamma$ when $\tau$ is no more than
    \begin{equation}
        C
        \kappa^2 
        n^{7.5}
        \log^2 n 
        \left(
        \max\left\{1,\sqrt{\frac{1}{n}\log\frac{2M}{\gamma}}
        \right\}
        \right)^2
        \log\frac{2M}{\gamma}.
        \label{eq:mixing-time-with-lee-vempala-isoperimetry}
    \end{equation}
\end{theorem}

The difference of $1/2$ in the exponent of $n$ between \eqref{eq:mixing-time-with-lee-vempala-isoperimetry} and the Laddha and Vempala result
is due to the difference in concentration behavior between strongly log-concave distributions and the uniform distribution, which is log-concave but not strongly log-concave.

To prove Theorem \ref{thm:main-mixing-time}, we establish a conductance
bound in a high-probability region of $\mathbb{R}^n$, specifically, 
a mode-centered Euclidean ball of sufficiently large radius. Building on the proof technique
of Laddha and Vempala, we prove an
isoperimetric inequality for subsets of the Euclidean ball.
The crucial insight that enables us to do so is the fact that smooth functions
are approximately uniform on small-enough domains.
In order to state the isoperimetric inequality, we will need
the following definition.

\begin{definition}[Axis-disjoint sets]
    Two sets $S_1$ and $S_2$ are axis-disjoint if,
    for every pair of vectors $x, y$ where $x\in S_1$ and $y\in S_2$,
    $|\{i\in[n] \,\, s.t. \,\, x_i\neq y_i\}| \geq 2$.
\end{definition}

The Gibbs sampler cannot transition between two axis-disjoint sets in a single iteration.
The following isoperimetric inequality
gives a worst-case lower bound on the measure
of the region of the state space that \emph{is} accessible to the dynamics from either of two axis-disjoint regions, and
is the main technical contribution of this paper.

\begin{lemma}
\label{thm-l0-isoperimetry}
    Consider a probability density function $\pi(x)=Z^{-1}e^{-f(x)}$, $x\in\mathbb{R}^n$, where $f$ is $\mu$-strongly convex in $x$ and has $L$-Lipschitz
    gradients.
    Let $\Pi$ be the associated probability measure.
    Let $x^{\star}=\argmax_x \pi(x)$.
    For $\varepsilon\in(0,1/2)$, define the function
    \begin{equation*}
        r(\varepsilon)=2+2\max\left\{\left(\frac{1}{n}\log\frac{1}{\varepsilon}\right)^{1/4},
        \left(\frac{1}{n}\log\frac{1}{\varepsilon}\right)^{1/2}
        \right\}.
    \end{equation*}
    Consider a Euclidean ball $K$ centered at $x^{\star}$ 
    with radius strictly greater than $r(\varepsilon)\sqrt{n/\mu}$ and no larger than twice that quantity.
    Let $S_1\cup S_2\cup S_3$ be a partition of $K$ such that $S_1$ and $S_2$ are
    axis-disjoint.
    Then there is a positive quantity $\Psi$ such that 
    \begin{equation}
        \Pi(S_3)\geq\Psi\min\left\{\frac{\Pi(S_1)}{5}-\varepsilon,\frac{\Pi(S_2)}{5}-\varepsilon\right\},
    \end{equation}
    where
    \begin{equation}
    \Psi
    \geq
    C\frac{1}{\kappa n^{2+\frac{3}{4}}
    \log n
    \max\left\{1,\sqrt{\frac{1}{n}\log\frac{1}{\varepsilon}}
    \right\}}
    \end{equation}
    for some universal constant $C$.
\end{lemma}

In independent and contemporaneous work, Ascolani, Lavenant, and Zanella 
proved a per-iteration contraction of the Kullback-Leibler 
divergence between the law of the sampler and the target under the dynamics of the Gibbs sampler 
for the same class of target distributions as in Theorem \ref{thm:main-mixing-time}
\cite{ascolani2024entropycontractiongibbssampler}.
Their analysis implies linear dependence of the mixing time on both the condition number and $n$.
For several technical reasons, it is unclear how to directly compare this with Theorem \ref{thm:main-mixing-time}.
One of these is that the proof of Theorem \ref{thm:main-mixing-time} fundamentally treats the sampler as being constrained to a convex subset of $\mathbb{R}^n$. 
This is not the case in \cite{ascolani2024entropycontractiongibbssampler}.
We leave a careful comparison of the two types of results as an interesting 
open problem.

The remainder of the paper is devoted to the proofs of Theorem \ref{thm:main-mixing-time} and Lemma \ref{thm-l0-isoperimetry},
followed by a discussion.

\section{Preliminaries}

Denote by $\pi(x):\mathbb{R}^n\rightarrow\mathbb{R}_{+}$ the normalized 
density function of a distribution that can be written in the form $e^{-f(x)}/Z$
where $f$ is $\mu$-strongly convex and $L$-smooth. That is, for any
$x,y\in\mathbb{R}^n$ and some $L\geq\mu>0$, $f$ satisfies the following inequalities:
\begin{gather}
    f(y)\geq f(x) + \langle \nabla f(x), y-x\rangle + \frac{\mu}{2}||x-y||^2,\\
    f(y)\leq f(x) + \langle \nabla f(x), y-x\rangle + \frac{L}{2}||x-y||^2.
\end{gather}
$Z$ is the partition function $\int_{\mathbb{R}^n} e^{-f(x)}\,dx$.
$\pi$ is absolutely continuous with respect to the Lebesgue measure, so
we can and will safely replace $d\pi(x)$ with $\pi(x)dx$ in all integrals.

Let $x^{\star}=\argmax_{x\in\mathbb{R}^n}\pi(x)$
be the mode of $\pi$. Define the condition number $\kappa=L/\mu$.

We write $\Pi$ for the probability measure associated with $\pi$, and $\Pi_{n-1}$ for the measure induced by $\Pi$ on 
an $n-1$-dimensional set.
We write $\vol$ for the Lebesgue measure of a set.

Superscripts on $x$ indicate iteration number and subscripts
index vector components.
$[n]$ denotes the set of natural numbers up to and including $n$.
Let $e_i$, $i\in[n]$, denote a unit vector in the $i$th coordinate direction.
We use the notation $x_{-i}$ to represent the $n-1$-dimensional vector
that results from leaving out the $i$th coordinate of $x$.

We are now ready to formally describe the algorithm.

\begin{algorithm}
\caption{Gibbs sampler}
    \begin{algorithmic}
    \STATE Given $x^0\in\mathbb{R}^n$, $\tau\in\mathbb{N}$:
        \FOR{t=1\dots $\tau$}
        \STATE Choose $i\in[n]$ uniformly at random.
        \STATE Draw a sample $y$ from $\pi(x_i|x_{-i})$.
        \STATE $x^{t}_i\rightarrow y$.
        \ENDFOR
    \end{algorithmic}
    \label{alg:gibbs}
\end{algorithm}

The stationary distribution of Algorithm \ref{alg:gibbs} is $\pi$.

In preparation for the proof of Theorem \ref{thm:main-mixing-time},
we give some basic facts about discrete-time Markov chains in 
continuous state spaces.

A discrete-time Markov chain on 
$\mathbb{R}^n$
is specified by an initial state $x^0\in\mathbb{R}^n$ and a conditional 
distribution $P_{x}(A)$. For any measurable subset $A$ of $\mathbb{R}^n$, 
$P_x(A)$ gives the probability of transitioning to a state $y\in A$ 
given that the chain is currently in the state $x$.
A distribution $\pi$ is \emph{stationary} for this chain if $P_x(A)$
is given by $\Pi(A)$ when $x$ is distributed according to $\pi$:
\begin{equation}
    \int_{\mathbb{R}^n} P_x(A)\pi(x)dx\, = \Pi(A).
\end{equation}

A Markov chain is \emph{reversible} if the probability of a transition between
two states is equally likely in both directions.
That is, for any two measurable
sets $A,B$, we must have
\begin{equation}
    \int_{B}P_x(A)\pi(x)dx = \int_{A}P_x(B)\pi(x)dx.
    \label{eq:reversibility}
\end{equation}
Algorithm \ref{alg:gibbs} is reversible. 

The \emph{ergodic flow} $p(A)$ of a set $A$ measures how likely the dynamics is
to exit it in a single step:
\begin{equation}
    p(A)=\int_A P_x\left(\mathbb{R}^n{\setminus}A\right)\pi(x)dx.
\end{equation}
The ratio of the ergodic flow of a set to either its measure 
or the measure of its complement, whichever is smaller, is its \emph{conductance} $\phi(A)$:
\begin{equation}
    \phi(A)=\frac{p(A)}{\min\{\Pi(A),\Pi(\mathbb{R}^n{\setminus}A)\}}.
    \label{eq:defn-of-conductance-of-a-set}
\end{equation}
The conductance $\phi$ \emph{of the chain} is the infimum over all 
measurable sets of their conductances. It is convenient to restrict 
to sets of measure not exceeding a half so we may drop the minimum in \eqref{eq:defn-of-conductance-of-a-set}. Then the conductance is given by the following 
expression:
\begin{equation}
    \phi \equiv \inf_{A:\,\Pi(A)\leq 1/2} \frac{p(A)}{\Pi(A)}.
    \label{eq:defn-of-conductance}
\end{equation}
The conductance quantifies bottlenecks to the dynamics. If it is
small, there is at least one region of the state space that is not easy
to escape from, limiting the ability of the sampler to explore the state
space. 
In what follows, we will work with a weakening of the definition
of the conductance, called the \emph{s-conductance}, which allows us to neglect
sets of measure no more than $s$ in the infimum in \eqref{eq:defn-of-conductance}.
For any $0<s<1/2$, we have
\begin{equation}
    \phi_s \equiv \inf_{A:\,s<\Pi(A)\leq 1/2} \frac{p(A)}{\Pi(A)-s}.
\end{equation}

At any specific iteration number $t$, the current state of the chain,
$x^t$, is distributed according to some distribution $\pi^t$.
We will refer to $\pi^t$ as \emph{the law of the sampler at time $t$}.
We use the total variation distance
to measure the distance between the law of the sampler and its stationary
distribution, defined as follows:
\begin{equation}
    d_{TV}\left(\pi^t, \pi\right)=\sup_{A}\left|\pi^t(A)-\pi(A)\right|.
    \label{eq:defn-TV-distance}
\end{equation}
The supremum in \eqref{eq:defn-TV-distance} is taken over all measurable sets.

The \emph{mixing time} of the chain is the smallest 
number of iterations $\tau(\epsilon)$ needed to drive this distance 
below some small $\epsilon$:
\begin{equation}
    \tau(\epsilon)\equiv\inf\{t\in\mathbb{N} \,\, s.t. \,\, d_{TV}(\pi^t, \pi)\leq\epsilon\}.
\end{equation}

We will work with a \emph{lazy} version of Algorithm \ref{alg:gibbs}, where
at every iteration, with probability $1/2$ we do nothing. This ensures
that $\pi$ is the \emph{unique} stationary distribution for Algorithm \ref{alg:gibbs}.

The last basic notion we shall need is one that characterizes the quality
of the initial distribution $\pi^0$. 
$\pi^0$ is a \emph{warm start}
for $\pi$ if, for some $M>0$,
\begin{equation}
    \sup_{A} \frac{\pi^0(A)}{\pi(A)}\leq M,
\end{equation}
where the supremum is taken over all measurable sets.
A result due to Altschuler and Chewi \cite{altschuler-chewi2024-warm-start} guarantees 
that a warm start to any log-concave distribution can be computed in 
$O(\sqrt{n})$ iterations. 
Since this is dwarfed by the mixing time in Theorem \ref{thm:main-mixing-time},
the warm start is a mild assumption to make for Gibbs sampling from log-concave
distributions.

Our goal is to upper bound $\tau(\epsilon)$ for Algorithm \ref{alg:gibbs}.
In particular, we are interested in identifying the $n$-dependence of $\tau(\epsilon)$.
To do so, we will lower bound the s-conductance and then invert the following 
result due to Lov\'{a}sz and Simonovits \cite{lovasz-simonovits-1993} to upper bound the mixing time.
\begin{theorem}\label{thm:lovasz-simonovits}
Consider a lazy, reversible Markov chain. Let $s\in(0,1/2)$. If the initial distribution is
$M$-warm with respect to the stationary distribution, then after $t$ iterations,
    \begin{equation}
    d_{TV}(\pi^t,\pi)
    \leq Ms + M\left(1-\frac{\phi^2_s}{2}\right)^t.
    \label{eq:tv-distance-thm}
    \end{equation}
\end{theorem}
Theorem \ref{thm:lovasz-simonovits} is itself a generalization to arbitrary state
spaces of a result due to Jerrum and Sinclair \cite{jerrum-sinclair-1988}.
Both results are related to Cheeger's inequality \cite{cheeger1970lower} in differential
geometry (see \cite{lovasz1993random} for a treatment of the relationship).
The presence of the warmness parameter $M$ in \eqref{eq:tv-distance-thm} 
is the price to be paid for working not with the conductance 
but with the s-conductance.

\section{Upper bounding the mixing time}
\label{sec:conductance-bounds}

The proof of Theorem \ref{thm:main-mixing-time}, presented in this section, proceeds 
via a lower bound on the $s$-conductance of the Gibbs sampler inside a convex high-probability region within which the gradient norm of $f$ can be bounded.
The argument illustrates the use of Lemma \ref{thm-l0-isoperimetry}, which is proved in the following section.

See Figure \ref{fig:conductance-argument} for an illustration of the overall construction.

We will need the following concentration result for strongly log-concave distributions.

\begin{lemma}
    \label{lemma:dwivedi-concentration}
    Let $\Pi$ be a $\mu$-strongly log-concave measure on $\mathbb{R}^n$
    and $\pi$ the associated density.
    Let $x^{\star}=\argmax_{x\in\mathbb{R}^n}\pi(x)$.
    Let $\mathcal{B}(x,r)$ denote a Euclidean ball of radius $r$ centered at $x$.
    For any $\varepsilon\in (0,1/2)$ and
    \begin{equation}
    r(\varepsilon)=2+2\max\left\{\left(\frac{1}{n}\log\frac{1}{\varepsilon}\right)^{1/4},
    \left(\frac{1}{n}\log\frac{1}{\varepsilon}\right)^{1/2}
    \right\},
    \label{eq:defn-of-r}
    \end{equation}
    we have
    \begin{equation}
    \Pi\left(\mathcal{B}\left(x^{\star},r(\varepsilon)\sqrt{\frac{n}{\mu}}\right)\right)\geq 1-\varepsilon.
    \end{equation}
\end{lemma}
This lemma is due to Dwivedi, Chen, Wainwright, and Yu \cite{dwivedi19-metropolis}.

Let $s\in(0,1/2)$. Take $\varepsilon=s/11$ in Lemma \ref{lemma:dwivedi-concentration}, and
consider a Euclidean ball $K$ centered ar $x^{\star}$ with radius $R$ such that 
\begin{equation}
    r\left(\frac{s}{11}\right)\sqrt{\frac{n}{\mu}}<R\leq2r\left(\frac{s}{11}\right)\sqrt{\frac{n}{\mu}},
\end{equation}
where $r$ is the function defined in \eqref{eq:defn-of-r}. (The reason for defining $R$ thus via two inequalities will be made clear in the proof of Lemma \ref{thm-l0-isoperimetry}.)
Then by Lemma \ref{lemma:dwivedi-concentration}, $\Pi(K)>1-s/11$
and $\Pi(\mathbb{R}^n{\setminus}K)<s/11$.

The following lemma enables us to construct a pair of axis-disjoint subsets of $\mathbb{R}^n$.

\begin{lemma}\label{lemma:axis-disjoint-sets}
    Let $A_1\cup A_2$ be a partition of $\mathbb{R}^n$, and let
    $A_1^{\prime}=\{x\in A_1 \,\,s.t.\,\, P_x(A_2)<\frac{1}{2n}\}$ and
    $A_2^{\prime}=\{x\in A_2 \,\,s.t.\,\, P_x(A_1)<\frac{1}{2n}\}$.
    Then $A_1^{\prime}$ and $A_2^{\prime}$ are axis-disjoint.
\end{lemma}
\begin{proof}
    The proof is by contradiction and relies on the fact that the probability of
    finding $x^{t+1}$ along each of the $n$ lines that pass through $x^t$ and run  
    parallel to the coordinate axes is $1/n$.
    
    Assume that $A_1^{\prime}$ and $A_2^{\prime}$ are not axis-disjoint.
    Then there must be a line running parallel to one of the coordinate axes that 
    passes through both sets. Without loss of generality, let this line, call it $\ell_j$, run along
    the coordinate axis $e_j$.
    Consider a point $y\in \ell_j\cap A_1^{\prime}$ and a point $z\in\ell_j \cap A_2^{\prime}$.
    By construction, we have
    \begin{equation*}
        P_{y}(\ell_j\cap A_2)\leq P_{y}(A_2) < \frac{1}{2n},~\text{and}~
        P_{z}(\ell_j\cap A_1)\leq P_{z}(A_1) < \frac{1}{2n}.
    \end{equation*}
    This implies
    \begin{equation*}
        P_{y}(\ell_j\cap A_2)+P_{z}(\ell_j\cap A_1)<\frac{1}{n}.
    \end{equation*}
    Furthermore,
    \begin{equation*}
        P_{y}(\ell_j\cap A_2)+P_{z}(\ell_j\cap A_1)
        =\Pi(\ell_j\cap A_2|y_{-j})+\Pi(\ell_j\cap A_1|z_{-j})=\Pi(\ell_j)=1/n.
    \end{equation*}
    In the second 
    equality we have used the fact that conditioning on $z_{-j}$ is equivalent to
    conditioning on $y_{-j}$ because $y$ and $z$ differ only in their $j^{th}$ coordinate.
    Thus we arrive at the contradiction $1<1$.
\end{proof}

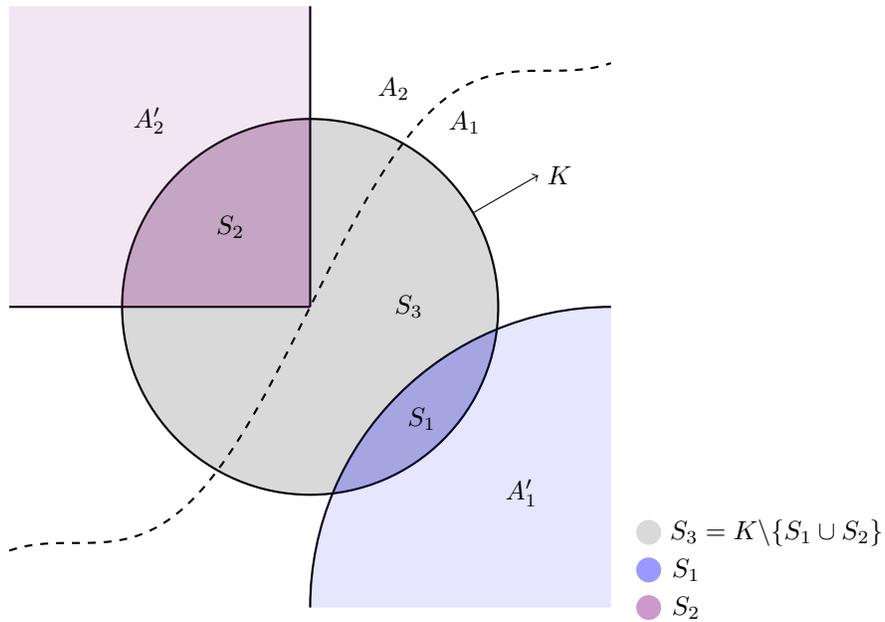
\begin{figure}
\centering
\begin{tikzpicture}
    \draw[thick, name path = circ, fill=gray!30] (0,0) circle(2.5);
    
    \draw[thick, dashed, domain=-4:4, samples=100] plot (\x, {sin(\x r) + \x});
    
    \coordinate (A) at (180:4);
    \coordinate (B) at (90:4);

    \coordinate (C) at (0:4);
    \coordinate (D) at (275:4);

    \coordinate (E) at (135:1.5); 
    \coordinate (F) at (315:2.1); 

    \coordinate (G) at (55:3); 
    \coordinate (H) at (75:3); 

    \coordinate (I) at (315:3.5); 
    \coordinate (J) at (135:3.5); 

    \coordinate (K) at (0: 1);

    \draw[->] (30:2.5) -- (30:3.5);
    \node at (30:3.5) [right] {$K$};

    \node [circle, fill=violet, opacity = 0.4] at (4.5,-4) {};
    \node at (5,-4) {$S_2$};
    \node [circle, fill=blue, opacity = 0.4] at (4.5,-3.5) {};
    \node at (5,-3.5) {$S_1$};
    \node [circle, fill=gray, opacity = 0.3] at (4.5,-3) {};
    \node at (6.2,-3) {$S_3=K{\setminus}\{S_1\cup S_2\}$};

    \draw[thick, name path=wholearc] (4,0) arc (90:180:4); 
    \draw[fill=blue, opacity=0.1, draw=none] (4,-4) -- (4,0) arc(90:180:4) -- (4,-4); 


    \draw [draw=none, name path = bottomarc, intersection segments={of=wholearc and circ,sequence={R2}}];
    \draw [draw=none, name path = toparc, intersection segments={of=circ and wholearc ,sequence={R2}}];

    \tikzfillbetween[of=toparc and bottomarc]{blue, opacity=0.15};

    \draw[-, thick, name path = vertline] (0,0) to (B);
    \draw[-, thick, name path = horline] (0,0) to (A);
    \draw[fill=violet, opacity = 0.1, draw=none] (0,0) -- (B) -- (-4,4) -- (A) -- (0,0);
    \draw[fill=violet, opacity = 0.15, draw=none] (0,0) -- (90:2.5) arc(90: 180: 2.5) -- (0,0);
    
    \node at (E) {$S_2$};
    \node at (F) {$S_1$};

    \node at (G) [right] {$A_1$};
    \node at (H) [right] {$A_2$};

    \node at (I) [right] {$A_1^{\prime}$};
    \node at (J) [right] {$A_2^{\prime}$};

    \node at (K) [right] {$S_3$};
    
\end{tikzpicture}
    \caption{Illustrated here is a partition $A_1\cup A_2$ of $\mathbb{R}^2$ 
    such that for some $s\in(0,1/2)$, $s<\Pi(A_1)\leq 1/2$. The Euclidean ball $K$ is centered at the mode $x^{\star}$ of $\pi$ and is large enough so that its measure differs from unity by a fraction of $s$.
    $A_1^{\prime}$ (in light pink) and $A_2^{\prime}$ (in light blue) are subsets of $A_1$ and $A_2$, respectively. $A_1^{\prime}$ and $A_2^{\prime}$ are axis disjoint.
    $S_1\cup S_2 \cup S_3$ is a partition of $K$ such that $S_i=K\cap A_i^{\prime}$ for $i=1,2$. $\Pi(S_1)\leq \Pi(S_2)$.}
    \label{fig:conductance-argument}
\end{figure}

Consider a partition $A_1\cup A_2$ of $\mathbb{R}^n$ such that $s<\Pi(A_1)\leq 1/2$, and 
subsets $A_1^{\prime}$ and $A_2^{\prime}$ defined as in Lemma \ref{lemma:axis-disjoint-sets}. $A_1^{\prime}$ and $A_2^{\prime}$ are axis-disjoint.
We have
\begin{gather}
    \Pi(A_i)=\Pi(A_i{\setminus} K)+\Pi(A_i\cap K)
    <\frac{s}{11}+\Pi(A_i\cap K),\nonumber\\
    \Rightarrow 
    \Pi(A_i\cap K)> \Pi(A_i)-\frac{s}{11}
    \label{eq:measure-outside-A-in-K}
\end{gather}
for $i=1, 2$.

Let $S_i=K\cap A_i^{\prime}$ for $i=1, 2$ and $S_3=K{\setminus}S_1{\setminus}S_2$. $S_1$ and $S_2$ are axis-disjoint.

We write down two lower bounds on the ergodic flow of $A_1$.
\begin{align}
    p(A_1)
    &=\int_{A_1}P_x(A_2)\,\pi(x)dx\nonumber\\
    &=\int_{A_1{\setminus}K}P_x(A_2)\,\pi(x)dx
    +\int_{S_1}P_x(A_2)\,\pi(x)dx
    +\int_{A_1\cap K{\setminus}S_1}P_x(A_2)\,\pi(x)dx\nonumber\\
    &\geq \int_{A_1\cap K{\setminus} S_1}P_x(A_2)\,\pi(x)dx\nonumber\\
    &\geq \frac{1}{2n}\Pi(A_1\cap K{\setminus}S_1).
    \label{eq:A1-ergodic-flow-1}
\end{align}
In the fourth line we have applied the definition of $A_1^{\prime}$.
Similarly, 
\begin{align}
    p(A_1)
    &=\int_{A_1}P_x(A_2)d\pi(x)\nonumber\\
    &=\int_{A_1}\left[
    P_x(A_2{\setminus}K)
    +P_x(S_2)
    +P_x(A_2\cap K{\setminus}S_2)
    \right]\,\pi(x)dx
    \nonumber\\
    &\geq \int_{A_1}P_x(A_2\cap K{\setminus} S_2)\,\pi(x)dx\nonumber\\
    &=\int_{A_2\cap K{\setminus} S_2}P_x(A_1)\,\pi(x)dx\nonumber\\
    &\geq \frac{1}{2n}\Pi(A_2\cap K{\setminus}S_2).
    \label{eq:A1-ergodic-flow-2}
\end{align}
In the penultimate line above, we have applied the reversibility property \eqref{eq:reversibility} of Algorithm \ref{alg:gibbs}. In the last line, we have applied the definition of $A_2^{\prime}$.

There are two cases to consider depending on the relative sizes of $S_i$ in $A_i\cap K$, $i=1,2$.

\paragraph{Case 1.}
If $\Pi(S_1)<\Pi(A_1\cap K)/2$,
then $\Pi(A_1\cap K{\setminus}S_1)\geq\Pi(A_1\cap K)/2$. 
From \eqref{eq:A1-ergodic-flow-1} we have
\begin{equation}
    p(A_1)
    \geq\frac{1}{2n}\Pi(A_1\cap K{\setminus}S_1)
    \geq \frac{1}{4n}\Pi(A_1\cap K)
    >\frac{1}{4n}\left(\Pi(A_1)-\frac{s}{11}\right)
    >\frac{1}{4n}\left(\Pi(A_1)-s\right).
\end{equation}
Similarly, if $\Pi(S_2)<\Pi(A_2\cap K)/2$, 
then $\Pi(A_2\cap K{\setminus}S_2)\geq\Pi(A_2\cap K)/2$, and
\eqref{eq:A1-ergodic-flow-2} gives the bound
\begin{equation}
    p(A_1)\geq\frac{1}{2n}\Pi(A_2\cap K{\setminus}S_2)
    \geq \frac{1}{4n}\Pi(A_2\cap K)
    >\frac{1}{4n}\left(\Pi(A_2)-\frac{s}{11}\right)
    \geq\frac{1}{4n}\left(\Pi(A_1)-\frac{s}{11}\right)
    \geq\frac{1}{4n}\left(\Pi(A_1)-s\right),
\end{equation}
where the penultimate inequality follows by construction.

Thus if either either $S_1$ or $S_2$ is small with respect to $A_1\cap K$ or $A_2\cap K$, respectively, 
the ergodic flow of $A_1$ is lower bounded by $\frac{1}{4n}\left(\Pi(A_1)-s\right)$.

\paragraph{Case 2.}
When both $S_1$ and $S_2$ 
are large in measure with respect to $A_1\cap K$ and $A_2\cap K$, respectively, 
Lemma \ref{thm-l0-isoperimetry} is needed
to lower bound the measure of $K{\setminus}S_1{\setminus}S_2$.

In particular, if $\Pi(S_1)\geq\Pi(A_1\cap K)/2$ and $\Pi(S_2)\geq\Pi(A_2\cap K)/2$,
then, summing \eqref{eq:A1-ergodic-flow-1} and \eqref{eq:A1-ergodic-flow-2},
we have
\begin{align}
    p(A_1)
    &> \frac{1}{2} \left(\frac{1}{2n}\Pi(A_1\cap K{\setminus}S_1)+\frac{1}{2n}\Pi(A_2\cap K{\setminus}S_2)\right)\nonumber\\
    &=\frac{1}{4n}\Pi(S_3)\nonumber\\
    &>\frac{\Psi}{4n}\left(\frac{\Pi(S_1)}{5}-\frac{s}{11}\right)
    \geq \frac{\Psi}{4n}\left(\frac{1}{5}\frac{\Pi\left(A_1\cap K\right)}{2}-\frac{s}{11}\right)
    \geq \frac{\Psi}{40n}\left(\Pi(A_1)-s\right).
\end{align}
In the third line we have applied
Lemma \ref{thm-l0-isoperimetry} with $\varepsilon=s/11$.

This concludes discussion of Case 2.

We note that
\begin{equation}
    \min\left\{\frac{1}{4n},\frac{\Psi}{40n}\right\}=\frac{\Psi}{40n}\,\,\forall n\geq 2.
\end{equation}

Thus for any set $A_1\subset\mathbb{R}^n$ such that $s<\Pi(A_1)\leq 1/2$ for $s\in(0,1/2)$,
\begin{equation}
    \frac{p(A_1)}{\Pi(A_1)-s}>\frac{\Psi}{40n}.
\end{equation}
This implies the following lower bound on the the $s$-conductance of the Gibbs sampler:
\begin{equation}
    \phi_s> \frac{\Psi}{40n}.
    \label{eq:lower-bound-phis}
\end{equation}

We are now ready to prove Theorem \ref{thm:main-mixing-time}.

\begin{proof}[Proof of Theorem \ref{thm:main-mixing-time}]
   Making the assignment $s=\gamma/2M$ in \eqref{eq:tv-distance-thm}, we have
   \begin{equation}
       d_{TV}(\pi^{\tau},\pi)
       \leq\frac{\gamma}{2} + M\left(1-\frac{\phi_s^2}{2}\right)^{\tau}
       \leq\frac{\gamma}{2} + Me^{-\tau\phi_s^2/2}.
   \end{equation}
   To ensure $d_{TV}(\pi^{\tau},\pi)\leq\gamma$, it is therefore enough to 
   choose $\tau$ such that
   \begin{equation}
       2Me^{-\tau\phi_s^2/2}=\gamma 
       \Rightarrow \tau=\frac{2}{\phi_s^2}\log\frac{2M}{\gamma}.
   \end{equation}
   Applying the lower bound \eqref{eq:lower-bound-phis} to $\phi_s$, we have
   \begin{equation}
       \tau< 2^{5}10^{2}\frac{n^2}{\Psi^2}\log\frac{2M}{\gamma}.
   \end{equation}
\end{proof}

\section{Isoperimetry}
\label{sec:isoperimetry}

In this section we prove Lemma \ref{thm-l0-isoperimetry}.
The proof is bipartite.
In the first part, we write down an isoperimetric inequality on an 
$n$-dimensional cube of side $\delta$ contained in $K$. This inequality relies on
a similar result proved by Laddha and Vempala for a uniform
distribution on the cube, and on the fact that smooth
functions are approximately uniform on small domains.
$\delta$ must be chosen carefully to ensure a good approximation.
In the second part, we tile $S_1$ (see Figure \ref{fig:conductance-argument}) with these cubes to establish 
the lower bound on $\Psi$.

The proof of Lemma \ref{thm-l0-isoperimetry} is given in Section \ref{sec:isoperimetry-proof}.
In Section \ref{sec:approx-theory} we develop the necessary approximation
results.

\subsection{Approximation theory}
\label{sec:approx-theory}

$K$ is a Euclidean ball centered at $x^{\star}$. For $\varepsilon\in(0,1/2)$,
the radius of $K$ is lower and upper bounded as follows:
\begin{equation}
    r(\varepsilon)\sqrt{\frac{n}{\mu}}<R\leq 2r(\varepsilon)\sqrt{\frac{n}{\mu}},
    \label{eq:defn-of-R}
\end{equation}
where $r$ is the function defined in Lemma \ref{lemma:dwivedi-concentration}.

We work with cubes that are \emph{axis-aligned}:

\begin{definition}[Axis-aligned cubes.]
    An axis-aligned cube of side $\delta$ is the set 
    $\{x\in\mathbb{R}^n: \, ||x-x_0||_{\infty}\leq\delta\}$, 
    where $x_0$ is any reference point in the cube.
\end{definition}

\subsubsection{Control of $f$ in a cube}

Consider an $n$-dimensional axis-aligned cube $C\subset K$ of side $\delta$.

For any two points $x,y\in C$,
Taylor's theorem guarantees the existence of a point $z$, also in $C$, such that 
\begin{equation*}
    f(y)=f(x)+\langle\nabla f(x), y-x\rangle + \frac{1}{2}(y-x)^{\top}\nabla^2f(z)(y-x).
\end{equation*}
Rearranging and applying the triangle inequality, we have
\begin{align}
    |f(y)-f(x)|
    &\leq ||\nabla f(x)||\,\,||y-x|| + \frac{1}{2}||\nabla^2 f(z)||_{\text{op}}\,\,||y-x||^2.
    \label{eq:cube-approx-inter1}
\end{align}

The maximal Euclidean distance between any two points in $C$
is bounded above by $\delta\sqrt{n}$. Smoothness guarantees
that the operator norm of $\nabla^2 f(z)$ is bounded by $L$ for any $z\in\mathbb{R}^n$.
Moreover, for any $x\in K$,
\begin{equation}
    ||\nabla f(x)||
    =||\nabla f(x)-\nabla f(x^{\star})||
    \leq L||x-x^{\star}||
    \leq 2Lr(\varepsilon)\sqrt{\frac{n}{\mu}},
\end{equation}
where $\varepsilon\in(0,1/2)$.
The last inequality follows from the upper bound \eqref{eq:defn-of-R} on the radius of $K$.
Thus we arrive at the bound
\begin{equation}
    |f(y)-f(x)|\leq 2\frac{L}{\sqrt{\mu}}\delta r(\varepsilon)n + \frac{1}{2}L\delta^2n.
    \label{eq:cube-approx-inter2}
\end{equation}

In order to control the right-hand side of \eqref{eq:cube-approx-inter2},
we will need $\delta$ to scale inversely in $n$ and $L$
with the appropriate powers.
We note that 
\begin{equation}
    \frac{r(\varepsilon)}{\max\left\{1,\sqrt{\frac{1}{n}\log\frac{1}{\varepsilon}}\right\}}\leq 4.
\end{equation}
For some positive constant $c$
and function $\sigma>0$, we make
the assignment
\begin{equation}
    \delta 
    = \left(8c \sqrt{\kappa L}n^{1+\sigma}
    \max\left\{1,\sqrt{\frac{1}{n}\log\frac{1}{\varepsilon}}\right\}
    \right)^{-1},
    \label{eq:defn-of-delta}
\end{equation}
in terms of which we have the bound
\begin{equation}
    |f(y)-f(x)|\leq 
    \frac{1}{cn^{\sigma}}
    + \frac{1}{128 c^2\kappa n^{1+2\sigma}
    \left(\max\left\{1,\sqrt{\frac{1}{n}\log\frac{1}{\varepsilon}}\right\}\right)^2}
    \leq
    \frac{1}{cn^{\sigma}}
    + \frac{1}{128 c^2 n^{1+2\sigma}}.
\end{equation}
$\sigma$ and $c$ must be large enough to ensure that $|f(y)-f(x)|$ is both small $\forall n\geq 2$ and
goes to zero as $n\rightarrow\infty$, but no larger than
necessary because 
the final mixing time bound will have $c^2n^{-2\sigma}$ dependence.
We will find it useful to enforce the specific
constraint
\begin{equation*}
    |f(y)-f(x)|\leq \log\frac{6}{5} \,\,\, \forall n\geq 2.
\end{equation*}
Taking $\sigma=\frac{\log\log n}{\log n}$ and $c=8$ is enough to do so.

We finally arrive at the bound
\begin{equation}
    |f(y)-f(x)|\leq
    \frac{1}{8\log n}
    + \frac{1}{2^{13} n\log n}
    \equiv
    \nu(n)
    \leq \log\frac{6}{5} \,\,\forall n\geq 2.
    \label{eq:defn-of-nu}
\end{equation}
for the difference in $f$ between any two points in $C$.
We have defined the function $\nu(n)$
for the sake of notational convenience.
We will omit writing its $n$-dependence in the sequel.

\eqref{eq:defn-of-nu} enables the following bound on how much the 
ratio $\pi(x)/\pi(y)$ deviates from unity within $C$. 

\begin{fact}\label{fact:bound-on-density-ratio-in-cube}
For any $x,y\in C\subset K$,
\begin{equation}
    \left|e^{-(f(x)-f(y))}-1\right|\leq e^{\nu} - 1.
    \label{eq:area-approx-intermediate}
\end{equation}
\end{fact}
\begin{proof}
If $f(x)\geq f(y)$, then \eqref{eq:defn-of-nu} implies 
$e^{f(x)-f(y)}\leq e^{\nu}$ and therefore
\begin{equation}
    \left|e^{-(f(x)-f(y))}-1\right|=1-e^{-(f(x)-f(y))}\leq 1-e^{-\nu}.
    \label{eq:density-ratio-bound-intermediate}
\end{equation}
If $f(x)<f(y)$, then \eqref{eq:defn-of-nu} implies $e^{-(f(x)-f(y))}\leq e^{\nu}$
and we have
\begin{equation}
    \left|e^{-(f(x)-f(y))}-1\right|= e^{-(f(x)-f(y))}-1\leq e^{\nu}-1.
\end{equation}

$1-e^{-\nu}=e^{\nu}-1$ for $\nu=0$, and $\forall\nu>0$,
\begin{equation*}
    \frac{d}{d\nu}(1-e^{-\nu})< \frac{d}{d\nu}(e^{\nu}-1),
\end{equation*}
which implies $1-e^{-\nu} < e^{\nu}-1$.

\end{proof}

In Sections \ref{sec:volume-approx} and \ref{sec:area-approx} 
we use Fact \ref{fact:bound-on-density-ratio-in-cube} 
to write down uniform approximations to $\Pi$
and $\Pi_{n-1}$ for subsets of $C$.

\subsubsection{A uniform approximation for $\Pi$
}
\label{sec:volume-approx}

We continue discussing an axis-aligned cube $C\subset K$ of side $\delta$.
Consider any $S\subseteq C$.
Let
\begin{equation}
    w\equiv \argmin_{x\in C}f(x) = \argmax_{x\in C}\pi(x).
\end{equation}
$C$ is a bounded set and $f$ is strongly convex, and so we are guaranteed that
$w$ exists.
We proceed to give an error bound for the uniform approximation
$\pi(w)\vol(S)$ to $\Pi(S)$.

Since $\pi(w)\geq \pi(x)\,\,\forall x\in C$, we have the simple upper bound
$\Pi(S)\leq\pi(w)\vol(S)$. For a lower bound on $\Pi(S)$, we
have
\begin{align*}
    \pi(w) \vol(S)-\Pi (S)
    &= \int_S dx \left(\pi(w)-\pi(x)\right)\\
    &\leq \pi(w)\int_S dx \left(1-e^{-(f(x)-f(w))}\right)\\
    &\leq \pi(w)\left(1-e^{-\nu}\right)\int_S dx\\
    &= \pi(w)(1-e^{-\nu})\vol(S).
\end{align*}
In the penultimate step above we have applied \eqref{eq:density-ratio-bound-intermediate}.

It will be useful to summarize the discussion so far in the following Fact:
\begin{fact}\label{fact:vol-pi-approximations-on-the-cube}
Let 
$w\equiv \argmin_{x\in C}f(x)$. For any $S\subseteq C\subset K$, we have
\begin{equation}
    \frac{\Pi(S)}{\pi(w)\vol(S)}\in [e^{-\nu},1],
\end{equation}
where $\nu$ is as defined in \eqref{eq:defn-of-nu}.
\end{fact}

We are not limited to choosing $w$ as a reference point 
in $C$. We do so for convenience, but any choice of reference point
is sufficient.

\subsubsection{Approximating $\Pi$ by $\delta\cdot\Pi_{n-1}$
}
\label{sec:area-approx}

Let $\beta$ be a facet of $C$
normal to the coordinate direction $e_i$ and let $\beta_a$
be the $n-1$ dimensional set in $C$ that results from translating
$\beta$ to the location $x_i=a$ in $C$. See Figure \ref{fig:extensions-close-up} for an illustration
in three dimensions. Let $\omega\subseteq \beta$. Let $\omega_a$ be the
translation of $\omega$ to $x_i=a$, and let $B$ be the $n$ dimensional
subset of $C$ that is swept out by $\omega$ as it is translated along $e_i$ 
from $\beta$ all the way to the opposite facet. 
We will call this the \emph{extension of $\omega$ in $C$}. Without loss of generality,
let $x_i=0$ in $\beta$.

\begin{figure}
    \centering
\begin{tikzpicture}[scale=2]
    \draw[thick] (0, 0, 0) -- (1, 0, 0) -- (1, 1, 0) -- (0, 1, 0) -- cycle; 
    \draw[thick] (0, 0, 1) -- (1, 0, 1) -- (1, 1, 1) -- (0, 1, 1) -- cycle; 
    \draw[thick] (0, 0, 0) -- (0, 0, 1); 
    \draw[thick] (1, 0, 0) -- (1, 0, 1); 
    \draw[thick] (1, 1, 0) -- (1, 1, 1); 
    \draw[thick] (0, 1, 0) -- (0, 1, 1); 

    \draw[] (0.5, 0, 0) -- (1, 0, 0) -- (1, 0, 1) -- (0.5, 0, 1) -- cycle; 
    \draw[] (0.5, 0.5, 1) -- (1, 0.5, 1) -- (1, 0.5, 0) -- (0.5, 0.5, 0) -- cycle; 
    \draw[thick] (0.5, 0, 0) -- (0.5, 0.5, 0); 
    \draw[thick] (0.5, 0, 1) -- (0.5, 0.5, 1); 

    \draw[->] (1, 0, 0) -- (1.5, 0, 0) node[right] {$e_1$};
    \draw[->] (0, 1, 0) -- (0, 1.5, 0) node[above] {$e_2$};
    \draw[->] (0, 0, 1) -- (0, 0, 2) node[right] {$e_3$};

    \fill[black] (0, 0, 0) circle (0.5pt); 
    \node[left] at (0, 0, 0) {$(0,0,0)$}; 
    \fill[black] (0, 0, 0.5) circle (0.5pt); 
    \node[left] at (0, 0, 0.5) {$(0,0,a)$}; 
    
    \fill[black] (0, 0, 1) circle (0.5pt); 
    \node[left] at (0, 0, 1) {$(0,0,\delta)$}; 

    \fill[gray, opacity=0.5] (0, 0, 0.5) -- (1, 0, 0.5) -- (1, 1, 0.5) -- (0, 1, 0.5) -- cycle;
    \fill[yellow, opacity=0.3] (0, 0, 0) -- (1, 0, 0) -- (1, 1, 0) -- (0, 1, 0) -- cycle;

    \fill[yellow, opacity=0.6] (1, 0, 0) -- (1, 0.5, 0) -- (0.5, 0.5, 0) -- (0.5, 0, 0) -- (1, 0, 0);
    \fill[gray, opacity=0.5] (1, 0, 0.5) -- (1, 0.5, 0.5) -- (0.5, 0.5, 0.5) -- (0.5, 0, 0.5) -- (1, 0, 0.5);


    \draw[->] (0.2, 1, 0) -- (0.4, 1.4, 0) node[right] {$\beta$};
    \draw[->] (0.3, 0.8, 0) -- (0.5, 1.2, 0) node[right] {$\beta_a$};
    
\end{tikzpicture}
    \caption{Pictured here is a cube of side $\delta$.
    $\beta$, shaded in light yellow, is a facet normal to the $e_3$ coordinate axis. $\beta_a$, shaded in gray, is the set that results from translating $\beta$ along the $e_3$ axis to $x_3=a$.
    $\omega$ is the subset of $\beta$ shaded in a darker yellow, and $\omega_a$
    is the subset of $\beta_a$ shaded in a darker gray.
    Outlined in black is the 
    extension $B$ of $\omega$ along $e_3$ in the cube.
    }
    \label{fig:area-approximations}
\end{figure}
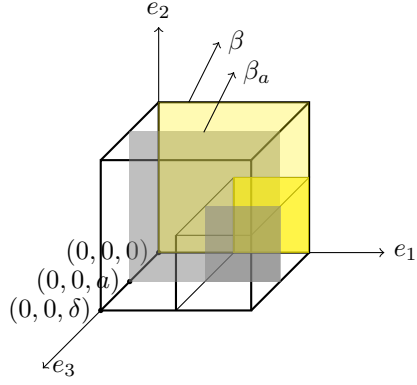

We can write $\Pi(B)$ as follows:
\begin{equation}
    \Pi(B) = \int_{0}^{\delta} dx_i \, \Pi_{n-1}(\omega_{x_i}).
\end{equation}
Denote by $x_i^a$ any vector with $i^{th}$ coordinate fixed to $a$.
We will now give an error bound for approximating
$\Pi(B)$ with $\delta\cdot\Pi_{n-1}(\omega_b)$ for any $x_i^b\in B$.

Fix any two values $a$ and $b$ of $x_i$ in $B$.
We have
\begin{align}
    \left|\Pi_{n-1}(\omega_a)-\Pi_{n-1}(\omega_b)\right|
    &=\left|\int_B dx_{-i} \, 
    \left[\pi(x_i^a)-\pi(x_i^b)\right]\right|\nonumber\\
    &\leq\int_B dx_{-i} \, \pi\left(x_i^b\right)
    \left|e^{-(f(x_i^a)-f(x_i^b))}-1\right|\nonumber\\
    &\leq (e^{\nu}-1)\int_B dx_{-i} \, \pi(x_i^b)\nonumber\\
    &=(e^{\nu}-1)\Pi_{n-1}(\omega_b),
    \label{eq:areas-are-approximately-equal}
\end{align}
where we have applied Fact \ref{fact:bound-on-density-ratio-in-cube}
in the penultimate line.
Therefore,
\begin{align}
    \left|\Pi(B) - \delta \cdot \Pi_{n-1}(\omega_b)\right|
    &=\left|\int_0^{\delta} dx_i \, \left[\Pi_{n-1}(\omega_{x_i})-\Pi_{n-1}(\omega_b)\right]\right|\nonumber\\
    &\leq \int_0^{\delta} dx_i \, 
    \left|\Pi_{n-1}(\omega_{x_i})-\Pi_{n-1}(\omega_b)\right|\nonumber\\
    &\leq (e^{\nu}-1) \Pi_{n-1}(\omega_b) \int_0^{\delta} dx_i
    \nonumber\\
    &=(e^{\nu}-1) \cdot \delta \cdot \Pi_{n-1}(\omega_b).
\end{align}

Since this bound holds for any value of $x^b_i$ in $C$, we are free to choose
$b$ to be the value of $x_i$ on one of the facets $\beta$ of the cube.
We have argued the following fact.

\begin{fact}\label{fact:area-approx}
Consider an axis-aligned cube $C\subset K$ of side $\delta$ and a log-smooth
strongly log-concave probability measure $\Pi$ with density $\pi$.
Let $\beta$ be any $n-1$-dimensional facet of $C$ and $\omega$ 
any subset of $\beta$. Let $B$ be the extension of $\omega$ in $C$ along the coordinate direction normal to $\beta$.
Then,
    \begin{equation}
    \left|\Pi(B) - \delta \cdot \Pi_{n-1}(\omega)\right|
    \leq(e^{\nu}-1) \cdot \delta \cdot \Pi_{n-1}(\omega),
    \end{equation}
    where $\nu$ is as defined in \eqref{eq:defn-of-nu}.
\end{fact}

\subsection{$L_0$ isoperimetry}
\label{sec:isoperimetry-proof}

\subsubsection{Cube isoperimetry}

The following lemma is due to Laddha and Vempala \cite{laddha-vempala}.

\begin{lemma}
\label{lemma:cube-isoperimetry-uniform}
    Consider an axis-aligned cube $C$ of side length $\delta$. 
    Let $S_1$ and $S_2$ be two axis-disjoint subsets of $C$ such that $\vol(S_1)\leq (2/3)\vol(C)$.
    Let $S_3=C{{\setminus}} \{S_1\cup S_2\}$. Then
    \begin{equation*}
        \vol(S_3)\geq \frac{\psi_c}{4}\vol(S_1)
    \end{equation*}
    where
    \begin{equation}
        \psi_c\geq\frac{\log 2}{\sqrt{n}}.
        \label{eq:cube-isoperimetric coefficient}
    \end{equation}
\end{lemma}
Laddha and Vempala originally
gave a lower bound of $n^{-1}\log 2$ for $\psi_c$. 
This was later improved to $n^{-1/2}\log 2$ by Fernandez 
\cite{fernandez2024ell0}, who also showed that the inverse square
root scaling of $n$ is optimal.

We prove a similar result for $\pi$ that relies on Lemma \ref{lemma:cube-isoperimetry-uniform}.

\begin{lemma}
\label{lemma:cube-isoperimetry-log-concave}
    Consider an axis-aligned cube $C$ of side length $\delta$
    and a log-smooth strongly log-concave measure $\Pi$ supported on $\mathbb{R}^n$ with associated density $\pi$.
    Let $S_1$ and 
    $S_2$ be two axis-disjoint subsets of $C$ such that 
    $$\Pi(S_1)\leq \frac{2}{3}e^{-\nu}\Pi(C),$$
    where $\nu$ is as defined in \eqref{eq:defn-of-nu}.
    Let $S_3=C{{\setminus}} \{S_1\cup S_2\}$. Then, 
    \begin{equation*}
        \Pi(S_3)\geq \frac{\psi_c}{4}e^{-\nu}\Pi(S_1).
    \end{equation*}
\end{lemma}
\begin{proof}
    Recall the quantity $w=\argmin_{x\in C}f(x)$.
    From Fact \ref{fact:vol-pi-approximations-on-the-cube} we 
    have the following two inequalities:
    \begin{align}
        &\pi(w)e^{-\nu}\vol(S_1)\leq\Pi(S_1),
        \label{eq:cube-lemma-inter1}\\
        &\pi(w)e^{-\nu}\vol(S_3)\leq\Pi(S_3).
        \label{eq:cube-lemma-inter4}
    \end{align}

    We check the condition in Lemma \ref{lemma:cube-isoperimetry-uniform}:
    \begin{align*}
        &\pi(w)e^{-\nu}\vol(S_1)
        \overset{(i)}{\leq}\Pi(S_1)
        \overset{(ii)}{\leq}\frac{2}{3}e^{-\nu}\Pi(C)
        \overset{(iii)}{\leq} \frac{2}{3}e^{-\nu}\pi(w)\vol(C)\\
        &\Rightarrow \vol(S_1)\leq\frac{2}{3}\vol(C).
    \end{align*}
    In $(i)$, we used \eqref{eq:cube-lemma-inter1}. $(ii)$ is by assumption. In $(iii)$ we used the definition of $w$.
    Thus the condition of Lemma \ref{lemma:cube-isoperimetry-uniform} is satisfied and we are guaranteed
    \begin{equation}
        \vol(S_3)\geq\frac{\psi_c}{4}\vol(S_1).
        \label{eq:cube-lemma-inter3}
    \end{equation}
    To translate this into the statement of the lemma, we apply 
    \eqref{eq:cube-lemma-inter3} to \eqref{eq:cube-lemma-inter4} and then use the definition of $w$. We have
    \begin{align*}
        &\frac{\Pi(S_3)}{\pi(w)}e^{\nu}\geq\vol(S_3)\geq\frac{\psi_c}{4}\vol(S_1)\geq \frac{\psi_c}{4}\frac{\Pi(S_1)}{\pi(w)}\\
        &\Rightarrow \Pi(S_3)\geq \frac{\psi_c}{4}e^{-\nu}\Pi(S_1).
    \end{align*}
\end{proof}

\subsubsection{Proof of Lemma \ref{thm-l0-isoperimetry}}

We begin by stating an isoperimetric theorem for strongly log-concave distributions that will be needed in the course of the argument.

\begin{theorem}
    Consider a convex body $A\subset\mathbb{R}^n$ and a subset $I\subseteq A$. 
    Let $\partial_A I$ be the internal boundary of $I$ in $A$.
    Let $\Pi$ be a strongly log-concave probability measure supported on $\mathbb{R}^n$
    and let $||\Sigma||_{op}$ be the operator norm (the largest eigenvalue) of its covariance 
    matrix $\Sigma$. Let $\Pi_{n-1}(\partial_K I)$ represent the measure of the boundary set
    $\partial_K I$ induced by $\Pi$.
    Then there exists a universal constant $c^{\prime}$ such that
    \begin{equation*}
        \Pi_{n-1}(\partial_A I)\geq\psi_{\pi}\min\left\{\Pi(I),\Pi(A{\setminus}I)\right\}
    \end{equation*}
    where
    \begin{equation}
    \psi_{\pi}\geq \frac{1}{c^{\prime} \left(n\,||\Sigma||^2_{op}\right)^{1/4}}.
    \label{eq:lee-vempala-logconcave-isoperimetry}
    \end{equation}
    \label{thm:log-concave-isoperimetry}
\end{theorem}
This theorem is due to Lee and Vempala \cite{lee-vempala-2024}.
For a $\mu$-strongly log-concave distribution, $||\Sigma||_{op}\leq \frac{1}{\mu}$.

We are now ready to prove Lemma \ref{thm-l0-isoperimetry}.

\begin{proof}[Proof of Lemma \ref{thm-l0-isoperimetry}]

Consider a Euclidean ball $K$ of radius $R$ in $\mathbb{R}^n$
embedded in a
grid of axis-aligned cubes of side 
\begin{equation*}
    \delta=\frac{1}{64\sqrt{\kappa L} n\log n \max\left\{1,\sqrt{\frac{1}{n}\log\frac{1}{\varepsilon}}\right\}}
\end{equation*}
for some $0<\varepsilon<1/2$.
$K$ is centered at the mode $x^{\star}$ of $\pi$.

$S_1$ and $S_2$ are two axis-disjoint subsets of $K$, and let
$S_3=K{\setminus}\{S_1\cup S_2\}$. Without loss of generality, let
$\Pi(S_1)\leq \Pi(S_2)$.

Let $K^{\prime}=(1-\alpha)K$ be the $\alpha$-shrinkage of $K$,
which we must consider in order to avoid overestimating the lower bound
on $\Pi(S_3)$.
We require $\alpha$ to be such that it can fit at least two cube 
diagonals, i.e., we need $\left(\sqrt{n}\delta\right)^{-1}\alpha\geq 2$.
This choice of ratio between $\alpha$ and $\delta$ ensures that any
cube intersecting $K^{\prime}$ as well as all its neighbors are contained in $K$.
We take
\begin{equation}
    \alpha=\frac{1}{4\sqrt{\kappa L} \sqrt{n}\log n \max\left\{1,\sqrt{\frac{1}{n}\log\frac{1}{\varepsilon}}\right\}}.
    \label{eq:alpha}
\end{equation}
$\alpha$ must be less than unity. In fact, we will enforce $\alpha\leq 1/2$
and place the following mild constraint on $L$ to ensure it
is satisfied:
\begin{equation}
    L>\max\left\{\frac{1}{n\log^2 n},\mu\right\}.
    \label{eq:constraint-on-L}
\end{equation}
Let $R^{\prime}$ be the radius of $K^{\prime}$. we make the assignment
\begin{equation}
    R^{\prime}=r(\varepsilon)\sqrt{\frac{n}{\mu}}.
    \label{eq:defn-of-Rprime}
\end{equation}
Together with $\alpha\leq 1/2$, \eqref{eq:defn-of-Rprime} implies 
$R^{\prime}<R\leq 2R^{\prime}$.
Moreover, Lemma \ref{lemma:dwivedi-concentration} guarantees
$\Pi(K^{\prime})\geq 1-\varepsilon$, and therefore
\begin{equation}
    \Pi(K{\setminus}K^{\prime})< \Pi(\mathbb{R}^n{\setminus}K^{\prime})<\varepsilon.
\end{equation}

Let $S_i^{\prime}=S_i \cap K^{\prime}$ for $i=1, 2$.

For any $X\subseteq K$, we have
\begin{equation}
    \Pi(X\cap K^{\prime})
    = \Pi(X) - \Pi(X\cap K{\setminus}K^{\prime})
    \geq \Pi(X) - \Pi(K{\setminus}K^{\prime})
    >\Pi(X)-\varepsilon.
    \label{eq:x-in-kprime-inequality}
\end{equation}

Let $C$ be the set of all cubes that intersect $S_1$. We partition it into
two sets, the boundary set $\bndryset$ and the bulk set $\bulkset$:
\begin{align}
    &\bndryset=\left\{c\in C \,\, s.t. \,\, \Pi(c\cap S_1)\leq \frac{2}{3}e^{-\nu}\Pi(c)\right\},\\
    &\bulkset=\left\{c\in C \,\, s.t. \,\, \Pi(c\cap S_1) > \frac{2}{3}e^{-\nu}\Pi(c)\right\}.
\end{align}
To preserve the bulk/boundary identities of these sets, we must have $\nu<\log\frac{4}{3}$. We have already ensured this (see the discussion surrounding \eqref{eq:defn-of-nu}).

We use curly notation, $\mathcal{C}_i$, to denote the $n$-dimensional body that is 
formed by taking the union of all the cubes in $C_i$ for $i=1, 2$.

Depending on whether most of the measure of 
$S_1$ is contained in the bulk $\bulksetbody$ or the boundary $\bndrysetbody$, there are two cases to consider.

\paragraph{Case 1.} $\Pi(\bndrysetbody\cap S_1)\geq\frac{1}{2}\Pi(S_1)$.

By definition, we have $\Pi(c\cap S_1)\leq \frac{2}{3}e^{-\nu}\Pi(c)$ 
for every $c\in \bndryset$, and so we can apply 
Lemma \ref{lemma:cube-isoperimetry-log-concave}
to each cube in $\bndryset$ and sum the resulting contribution over 
all the cubes in $\bndryset$ to lower bound $\Pi(S_3)$.
However, there being no guarantee that all cubes in $C_1$ are fully 
contained in $K$, we must correct for the possibility of overcounting 
contributions to $S_3$. We do so by considering only the subset of cubes
in $\bndryset$ that intersect $K^{\prime}$,
\begin{equation}
    \bndryset^{\prime}=\{c\in \bndryset \,\, s.t. \,\, c\cap K^{\prime}\neq\varnothing \}.
\end{equation}
Denote by $\mathcal{C}_1^{\prime}$ the solid body that is formed by taking the 
union of all the cubes in $\bndryset^{\prime}$.
Due to our choice of $\alpha$, we are guaranteed that $\mathcal{C}_1^{\prime} \subseteq K$.
By Lemma \ref{lemma:cube-isoperimetry-log-concave}, each cube $c$ in 
$C_1^{\prime}$ makes the following minimal contribution to the measure of $S_3$:
\begin{equation}
    \Pi(c\cap S_3)\geq\frac{\psi_c}{4}e^{-\nu}\Pi(c\cap S_1).
    \label{eq:case1-inter1}
\end{equation}
Summing \eqref{eq:case1-inter1} over all $c\in C_1^{\prime}$, we have
\begin{align*}
    \Pi(S_3)
    \geq \sum_{c\in C_1^{\prime}} \Pi(c\cap S_3)
    \geq \frac{\psi_c}{4}e^{-\nu}\sum_{c\in C_1^{\prime}}\Pi(c\cap S_1)
    =\frac{\psi_c}{4}e^{-\nu}\Pi(\mathcal{C}_1^{\prime}\cap S_1).
\end{align*}
To lower bound $\Pi(\mathcal{C}_1^{\prime}\cap S_1)$, 
note that while $\mathcal{C}_1^{\prime}$ may not be fully contained in $K^{\prime}$,
the set $\bndrysetbody\cap K^{\prime}$ is, and therefore
$\Pi(\mathcal{C}_1^{\prime})\geq\Pi(\mathcal{C}_1\cap K^{\prime})$.
Using this fact, we have
\begin{equation}    
    \Pi(\mathcal{C}_1^{\prime}\cap S_1)
    \geq\Pi(\mathcal{C}_1\cap S_1\cap K^{\prime})
    \geq\Pi(\mathcal{C}_1\cap S_1)-\varepsilon.
\end{equation}
The second inequality follows by taking $X=\mathcal{C}_1\cap S_1$ in 
\eqref{eq:x-in-kprime-inequality}.
Thus,
\begin{align}
    \Pi(S_3)
    \geq \frac{\psi_c}{4}e^{-\nu}
    \big(\Pi(\mathcal{C}_1\cap S_1)-\varepsilon\big)
    \geq \frac{\psi_c}{4}e^{-\nu}
    \left(\frac{1}{2}\Pi(S_1)-\varepsilon
    \right)
    > \frac{5\psi_c}{24}
    \left(\frac{1}{2}\Pi(S_1)-\varepsilon
    \right),
    \label{eq:case1-lower-bound}
\end{align}
where the second inequality follows from the condition $\Pi(\bndrysetbody\cap S_1)\geq\frac{1}{2}\Pi(S_1)$
and the third from $\nu<\log 6/5$.

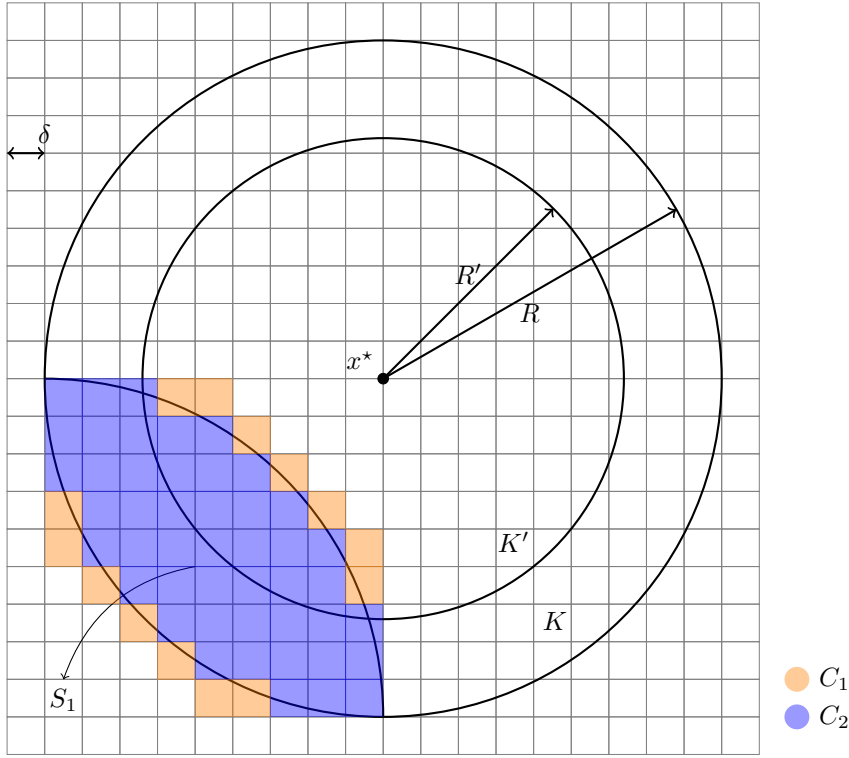
\begin{figure}
\centering
\begin{tikzpicture}[scale=1]

\def\squareSize{0.5}  

\def\gridWidth{19}  
\def\gridHeight{19} 

\foreach \x in {0, 1, 2, ..., \gridWidth} {
  \foreach \y in {0, 1, 2, ..., \gridHeight} {
    \node[anchor=center, opacity=0] at (\x*\squareSize+\squareSize/2, \y*\squareSize+\squareSize/2) {(\x,\y)}; 
    \draw[gray, very thin] (\x*\squareSize, \y*\squareSize) rectangle (\x*\squareSize+\squareSize, \y*\squareSize+\squareSize); 
  }
}

\coordinate (cen) at (5,5);
\draw[thick] (cen) circle (4.5);
\draw[thick] (cen) circle (3.2);

\begin{scope}[shift = {(cen)}]
    \coordinate (R) at (30:4.5);
    \coordinate (Rprime) at (45:3.2);
    \coordinate (K) at (300:4);
    \coordinate (Kprime) at (300:2.8);
\end{scope}
\draw[->, thick] (cen) -- (R) node [below,midway] {$R$};
\draw[->, thick] (cen) -- (Rprime) node [above,midway] {$R^{\prime}$};
\filldraw[black] (cen) circle (2pt) node [above left] {$x^{\star}$};

    \begin{scope}[shift={(0.5,0.5)}]
    \coordinate (fov) at (0,0) ;
    \coordinate (A) at (0:4.5) ;
    \coordinate (B) at (90:4.5) ;
    
    \coordinate (C) at (10, 0);
    \coordinate (Clabel) at (10.5, 0);
    \coordinate (D) at (10, 0.5);
    \coordinate (Dlabel) at (10.5, 0.5);

    \coordinate (S) at (2,2);
    \coordinate (Sarrowhead) at (0.25,0.5);
    \end{scope}

    \draw[->] (S) to[bend right] (Sarrowhead) node [below] {$S_1$};
    \node at (K) [above right] {$K$};
    \node at (Kprime) [above right] {$K^{\prime}$};
    \draw[<->, thick] (0,8) -- (0.5, 8) node [above] {$\delta$};

    \node [circle, fill=blue, opacity = 0.4] at (C) {};
    \node at (Clabel) {$C_2$};
    \node [circle, fill=orange, opacity = 0.4] at (D) {};
    \node at (Dlabel) {$C_1$};
    \draw[thick] (B) arc(90:0:4.5) ; 


\fill[blue, opacity=0.4] (1*\squareSize, 9*\squareSize) rectangle (2*\squareSize, 10*\squareSize);
\fill[blue, opacity=0.4] (1*\squareSize, 8*\squareSize) rectangle (2*\squareSize, 9*\squareSize);
\fill[blue, opacity=0.4] (1*\squareSize, 7*\squareSize) rectangle (2*\squareSize, 8*\squareSize);
\fill[orange, opacity=0.4] (1*\squareSize, 6*\squareSize) rectangle (2*\squareSize, 7*\squareSize);
\fill[orange, opacity=0.4] (1*\squareSize, 5*\squareSize) rectangle (2*\squareSize, 6*\squareSize);

\fill[blue, opacity=0.4] (2*\squareSize, 9*\squareSize) rectangle (3*\squareSize, 10*\squareSize);
\fill[blue, opacity=0.4] (2*\squareSize, 8*\squareSize) rectangle (3*\squareSize, 9*\squareSize);
\fill[blue, opacity=0.4] (2*\squareSize, 7*\squareSize) rectangle (3*\squareSize, 8*\squareSize);
\fill[blue, opacity=0.4] (2*\squareSize, 6*\squareSize) rectangle (3*\squareSize, 7*\squareSize);
\fill[blue, opacity=0.4] (2*\squareSize, 5*\squareSize) rectangle (3*\squareSize, 6*\squareSize);
\fill[orange, opacity=0.4] (2*\squareSize, 4*\squareSize) rectangle (3*\squareSize, 5*\squareSize);

\fill[blue, opacity=0.4] (3*\squareSize, 9*\squareSize) rectangle (4*\squareSize, 10*\squareSize);
\fill[blue, opacity=0.4] (3*\squareSize, 8*\squareSize) rectangle (4*\squareSize, 9*\squareSize);
\fill[blue, opacity=0.4] (3*\squareSize, 7*\squareSize) rectangle (4*\squareSize, 8*\squareSize);
\fill[blue, opacity=0.4] (3*\squareSize, 6*\squareSize) rectangle (4*\squareSize, 7*\squareSize);
\fill[blue, opacity=0.4] (3*\squareSize, 5*\squareSize) rectangle (4*\squareSize, 6*\squareSize);
\fill[blue, opacity=0.4] (3*\squareSize, 4*\squareSize) rectangle (4*\squareSize, 5*\squareSize);
\fill[orange, opacity=0.4] (3*\squareSize, 3*\squareSize) rectangle (4*\squareSize, 4*\squareSize);

\fill[orange, opacity=0.4] (4*\squareSize, 9*\squareSize) rectangle (5*\squareSize, 10*\squareSize);
\fill[blue, opacity=0.4] (4*\squareSize, 8*\squareSize) rectangle (5*\squareSize, 9*\squareSize);
\fill[blue, opacity=0.4] (4*\squareSize, 7*\squareSize) rectangle (5*\squareSize, 8*\squareSize);
\fill[blue, opacity=0.4] (4*\squareSize, 6*\squareSize) rectangle (5*\squareSize, 7*\squareSize);
\fill[blue, opacity=0.4] (4*\squareSize, 5*\squareSize) rectangle (5*\squareSize, 6*\squareSize);
\fill[blue, opacity=0.4] (4*\squareSize, 4*\squareSize) rectangle (5*\squareSize, 5*\squareSize);
\fill[blue, opacity=0.4] (4*\squareSize, 3*\squareSize) rectangle (5*\squareSize, 4*\squareSize);
\fill[orange, opacity=0.4] (4*\squareSize, 2*\squareSize) rectangle (5*\squareSize, 3*\squareSize);

\fill[orange, opacity=0.4] (5*\squareSize, 9*\squareSize) rectangle (6*\squareSize, 10*\squareSize);
\fill[blue, opacity=0.4] (5*\squareSize, 8*\squareSize) rectangle (6*\squareSize, 9*\squareSize);
\fill[blue, opacity=0.4] (5*\squareSize, 7*\squareSize) rectangle (6*\squareSize, 8*\squareSize);
\fill[blue, opacity=0.4] (5*\squareSize, 6*\squareSize) rectangle (6*\squareSize, 7*\squareSize);
\fill[blue, opacity=0.4] (5*\squareSize, 5*\squareSize) rectangle (6*\squareSize, 6*\squareSize);
\fill[blue, opacity=0.4] (5*\squareSize, 4*\squareSize) rectangle (6*\squareSize, 5*\squareSize);
\fill[blue, opacity=0.4] (5*\squareSize, 3*\squareSize) rectangle (6*\squareSize, 4*\squareSize);
\fill[blue, opacity=0.4] (5*\squareSize, 2*\squareSize) rectangle (6*\squareSize, 3*\squareSize);
\fill[orange, opacity=0.4] (5*\squareSize, 1*\squareSize) rectangle (6*\squareSize, 2*\squareSize);

\fill[orange, opacity=0.4] (6*\squareSize, 8*\squareSize) rectangle (7*\squareSize, 9*\squareSize);
\fill[blue, opacity=0.4] (6*\squareSize, 7*\squareSize) rectangle (7*\squareSize, 8*\squareSize);
\fill[blue, opacity=0.4] (6*\squareSize, 6*\squareSize) rectangle (7*\squareSize, 7*\squareSize);
\fill[blue, opacity=0.4] (6*\squareSize, 5*\squareSize) rectangle (7*\squareSize, 6*\squareSize);
\fill[blue, opacity=0.4] (6*\squareSize, 4*\squareSize) rectangle (7*\squareSize, 5*\squareSize);
\fill[blue, opacity=0.4] (6*\squareSize, 3*\squareSize) rectangle (7*\squareSize, 4*\squareSize);
\fill[blue, opacity=0.4] (6*\squareSize, 2*\squareSize) rectangle (7*\squareSize, 3*\squareSize);
\fill[orange, opacity=0.4] (6*\squareSize, 1*\squareSize) rectangle (7*\squareSize, 2*\squareSize);

\fill[orange, opacity=0.4] (7*\squareSize, 7*\squareSize) rectangle (8*\squareSize, 8*\squareSize);
\fill[blue, opacity=0.4] (7*\squareSize, 6*\squareSize) rectangle (8*\squareSize, 7*\squareSize);
\fill[blue, opacity=0.4] (7*\squareSize, 5*\squareSize) rectangle (8*\squareSize, 6*\squareSize);
\fill[blue, opacity=0.4] (7*\squareSize, 4*\squareSize) rectangle (8*\squareSize, 5*\squareSize);
\fill[blue, opacity=0.4] (7*\squareSize, 3*\squareSize) rectangle (8*\squareSize, 4*\squareSize);
\fill[blue, opacity=0.4] (7*\squareSize, 2*\squareSize) rectangle (8*\squareSize, 3*\squareSize);
\fill[blue, opacity=0.4] (7*\squareSize, 1*\squareSize) rectangle (8*\squareSize, 2*\squareSize);

\fill[orange, opacity=0.4] (8*\squareSize, 6*\squareSize) rectangle (9*\squareSize, 7*\squareSize);
\fill[blue, opacity=0.4] (8*\squareSize, 5*\squareSize) rectangle (9*\squareSize, 6*\squareSize);
\fill[blue, opacity=0.4] (8*\squareSize, 4*\squareSize) rectangle (9*\squareSize, 5*\squareSize);
\fill[blue, opacity=0.4] (8*\squareSize, 3*\squareSize) rectangle (9*\squareSize, 4*\squareSize);
\fill[blue, opacity=0.4] (8*\squareSize, 2*\squareSize) rectangle (9*\squareSize, 3*\squareSize);
\fill[blue, opacity=0.4] (8*\squareSize, 1*\squareSize) rectangle (9*\squareSize, 2*\squareSize);

\fill[orange, opacity=0.4] (9*\squareSize, 5*\squareSize) rectangle (10*\squareSize, 6*\squareSize);
\fill[orange, opacity=0.4] (9*\squareSize, 4*\squareSize) rectangle (10*\squareSize, 5*\squareSize);
\fill[blue, opacity=0.4] (9*\squareSize, 3*\squareSize) rectangle (10*\squareSize, 4*\squareSize);
\fill[blue, opacity=0.4] (9*\squareSize, 2*\squareSize) rectangle (10*\squareSize, 3*\squareSize);
\fill[blue, opacity=0.4] (9*\squareSize, 1*\squareSize) rectangle (10*\squareSize, 2*\squareSize);

\end{tikzpicture}
\caption{The Euclidean ball $K$ is embedded in a grid of cubes of side $\delta$.
Shaded in purple is the set of cubes in the bulk set $C_2$ covering $S_1$.
In orange is the boundary set $C_1$.}
\label{fig:grid-construction}
\end{figure}

\paragraph{Case 2.} $\Pi(\bndrysetbody\cap S_1)<\frac{1}{2}\Pi(S_1)$ 
$\Rightarrow$
$\Pi(\bulksetbody\cap S_1)\geq\frac{1}{2}\Pi(S_1)$.

Let $\partial_K\bulksetbody$ denote the internal boundary of 
$\bulksetbody$ in $K$, i.e., the intersection of $K$ and the boundary of $\bulksetbody$. Let $\mathcal{B}$ be the set of facets 
(of cubes) that intersect $\partial_K\bulksetbody$.
Consider a facet $\facet$ in $\mathcal{B}$ 
and the two cubes on either side of $\facet$. Since $\facet$ is on the
boundary of $\bulksetbody$, one of these cubes, call it $v$, is in 
$\bulkset$, and the other, $u$, is not. See Figure \ref{fig:extensions-close-up} for a picture of this construction.

In what follows, we will lower bound $\Pi(S_3\cap u)$, 
thus quantifying the minimal measure of the region to which the 
sampler can escape $S_1$ through
$\facet$.
The bound will be in terms of $\Pi(u)$.
We will rewrite it in terms of $\Pi_{n-1}(\facet)$,
gather contributions from all the facets in $\mathcal{B}$, and
then apply Theorem \ref{thm:log-concave-isoperimetry} 
to obtain a lower bound on $\Pi(S_3)$ in terms of $\Pi(S_1)$.

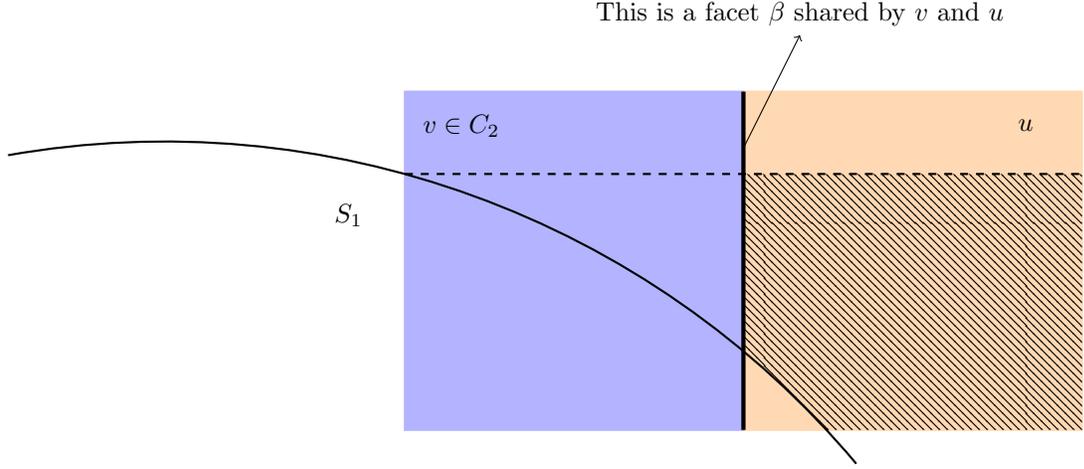
\begin{figure}
\centering
\begin{tikzpicture}[scale=1.5]
    
\draw[thick, fill, color=blue, opacity=0.3] (0, 0) rectangle (3, 3);

\draw[thick, fill, color=orange, opacity=0.3] (3, 0) rectangle (6, 3);

\draw[draw=none, thick, name path=vert] (0,0) -- (0,3);
\draw[ultra thick, name path=beta] (3,0) -- (3,3);

\draw[thick, name path=Sarc] (4, -0.3) arc[start angle=40, end angle=100, radius=8];

\path [name intersections={of=Sarc and vert,by=I}];
\draw[thick, dashed] (I) -- (6, |-I);

\path [name intersections={of=Sarc and beta,by=fv}];
\draw[draw=none, name path = bottomline] (3,0) -- (6,0);
\path [name intersections={of=Sarc and bottomline,by=fiv}];

\begin{scope}[shift = {(0, 0)}]
    \coordinate (slabel) at (-0.5, 1.9);
    \coordinate (facetlabelorigin) at (3, 2.5); 
    \coordinate (facetlabelend) at (3.5, 3.5); 
    \coordinate (lc) at (0.5, 2.7);
    \coordinate (rc) at (5.5, 2.7);
    \coordinate (lctext) at (0.8, -0.2);
    \coordinate (rctext) at (5.7, -0.2);
    \coordinate (rctextii) at (5.7, -0.45);
    \coordinate (rctextiii) at (5.7, -0.7);
    \coordinate (rctextiv) at (5.7, -0.95);
    \coordinate (rctextv) at (5.7, -1.2);
    \coordinate (rctextvi) at (5.7, -1.45);

    \coordinate (fi) at (3, |-I);
    \coordinate (fii) at (6, |-I);
    \coordinate (fiii) at (6,0);
\end{scope}

\fill[pattern=north west lines] (fi) -- (fii) -- (fiii) -- (fiv) -- (fv) -- cycle;

\draw[->] (facetlabelorigin) -- (facetlabelend) node [above] {This is a facet $\beta$ shared by $v$ and $u$};
\node at (lc) {$v\in C_2$};
\node at (rc) {$u$};
\node at (slabel) {$S_1$};

\end{tikzpicture}
\caption{The part of $\facet$ underneath the dashed line is $\mathcal{P}_{\facet}(S_1\cap v)$, the projection of $S_1\cap v$ onto $\beta$. The part of $u$ underneath the dashed line is $\curlyext$, the extension of $\mathcal{P}_{\facet}(S_1\cap v)$ in $u$ along the coordinate direction normal to $\beta$.
The subset of this region shaded with lines is the precisely the part of $S_3\cap u$ that the dynamics can escape to through $\beta$ from $S_1\cap v$.
A worst-case lower bound on its measure is derived in \eqref{eq:inter-bound-on-pi-S3-in-u}.}
\label{fig:extensions-close-up}
\end{figure}

Let $e_i$ be the coordinate vector perpendicular to $\beta$.
Let $\mathcal{P}_{\facet}(S_1\cap v)$ 
be the projection of $S_1\cap v$ onto $\facet$, and let 
$\mathcal{E}_u(\mathcal{P}_{\facet}(S_1\cap v))$
be the extension of this projection along $e_i$ in $u$. 
\begin{align}
    &\mathcal{P}_{\facet}(S_1\cap v)
    = \{x\in \facet \,\, s.t. \,\, \exists y\in v \,\, s.t. \,\, y_j=x_j \, \forall j\in[n]{\setminus} i\},\\
    &\mathcal{E}_u(\mathcal{P}_{\facet}(S_1\cap v))
    = \{x\in u \,\, s.t. \,\, \exists y\in \mathcal{P}_{\facet}(S_1\cap v) \,\, s.t. \,\, x_j=y_j \, \forall j\in[n]{\setminus} i\}.
\end{align}
$\mathcal{E}_u(\mathcal{P}_{\facet}(S_1\cap v))$
is precisely the region of $u$ that the dynamics can move to from 
$v$ in the direction $e_i$ in one step.
For notational simplicity, we will drop the argument 
of $\mathcal{P}_{\facet}$ in the sequel.

$S_1$ and $S_2$ are axis disjoint. Thus we have 
\begin{equation*}
    \curlyext\cap S_2=\varnothing.
\end{equation*}
Consequently,
\begin{gather}
    \Pi(\curlyext)
    =\Pi(\curlyext\cap S_1) + \Pi(\curlyext\cap S_3)
    \overset{(i)}{\leq} \Pi(u\cap S_1) + \Pi(u\cap S_3),\nonumber\\
    \Rightarrow\Pi(u\cap S_3)\geq\Pi(\curlyext)-\Pi(u\cap S_1).
    \label{eq:decomposition-of-pi-curly-ext-1}
\end{gather}
$(i)$ follows from the fact that $\curlyext\subseteq u$.

Next, we develop a lower bound on $\Pi(\curlyext)$ in terms of $\Pi(u)$.
By construction,
\begin{equation*}
    \vol(\curlyext)\geq\vol(S_1\cap v).
\end{equation*}
Let $w_v=\argmin_{x\in v}f(x)$. Then
\begin{equation*}
    \vol(S_1\cap v)\geq \frac{\Pi(S_1\cap v)}{\pi(w_v)}> \frac{2}{3}e^{-\nu}\frac{\Pi(v)}{\pi(w_v)},
\end{equation*}
where the second inequality follows from the fact that $v$ is an element of $\bulkset$.
Now, using Fact \ref{fact:vol-pi-approximations-on-the-cube},
\begin{equation*}
    \Pi(v)\geq e^{-\nu}\pi(w_v)\vol(v)=e^{-\nu}\pi(w_v)\vol(u).
\end{equation*}
Thus, 
\begin{equation*}
    \vol(S_1\cap v)\geq
    \frac{2}{3}e^{-2\nu}\vol(u),
\end{equation*}
i.e.,
\begin{equation}
    \vol(\curlyext)\geq \frac{2}{3}e^{-2\nu}\vol(u).
    \label{eq:bound-on-pi-ext-inter1}
\end{equation}
Let $w_{u}=\argmin_{x\in u}f(x)$. Then 
\begin{equation*}
    \vol(u)\geq \frac{\Pi(u)}{\pi(w_u)}
\end{equation*}
and by Fact \ref{fact:vol-pi-approximations-on-the-cube},
\begin{equation*}
    \frac{\Pi(\curlyext)}{\pi(w_u)}e^{\nu}\geq\vol(\curlyext).
\end{equation*}
Applying the last two bounds to \eqref{eq:bound-on-pi-ext-inter1},
we have the following lower bound on the measure of $\curlyext$:
\begin{equation}
    \Pi(\curlyext)\geq\frac{2}{3}e^{-3\nu}\Pi(u).
    \label{eq:lower-bound-on-pi-ext}
\end{equation}

Substituting \eqref{eq:lower-bound-on-pi-ext} in 
\eqref{eq:decomposition-of-pi-curly-ext-1}, 
we have the following lower bound on $\Pi(S_3\cap u)$ in terms of $\Pi(u)$ and
$\Pi(S_1\cap u)$:
\begin{equation}
    \Pi(u\cap S_3)
    \geq \frac{2}{3}e^{-3\nu}\Pi(u)-\Pi(u\cap S_1).
    \label{eq:decomposition-of-measure-of-curlyext}
\end{equation}

Now, if $u\notin \bndryset$ (i.e., $\Pi(u\cap S_1)=0$), the second term on the 
right-hand side of \eqref{eq:decomposition-of-measure-of-curlyext}
vanishes, and we are left with a bound on $\Pi(S_3\cap u)$ that depends 
only on $\Pi(u)$, as was our goal.
However, if $u\in \bndryset$, 
then there is more work to be done, which we split into 
two cases depending on the proportion of $u$ occupied by $S_1$. 
Let this proportion be represented by a quantity $h$, 
which we will choose later. Note that $h$ cannot exceed $\frac{2}{3}e^{-\nu}$.

\begin{itemize}
    \item 
    If $\Pi(u\cap S_1)\leq h\Pi(u)$, then from \eqref{eq:decomposition-of-measure-of-curlyext}, we have
\begin{equation}
    \Pi(u \cap S_3)
    \geq \Pi(\curlyext) - h\Pi(u)
    \geq \left(\frac{2}{3}e^{-3\nu}-h\right)\Pi(u).
    \label{eq:caseBi-lower-bound}
\end{equation}
    \item If $h\Pi(u)<\Pi(u\cap S_1)\leq\frac{2}{3}e^{-\nu}\Pi(u)$,
we can apply Lemma \ref{lemma:cube-isoperimetry-log-concave} to
$u\cap S_3$ to directly lower bound its measure:
\begin{equation}\label{eq:caseBii-lower-bound}
    \Pi(u\cap S_3)\geq \frac{\psi_c}{4}e^{-\nu}\Pi(S_1\cap u)\geq \frac{\psi_c}{4}e^{-\nu}h\Pi(u).
\end{equation}
\end{itemize}

\paragraph{The minimal bound on $\Pi(u\cap S_3)$.}

The contribution of $u$ to the measure of $S_3$ 
is at least the minimum of the three bounds we have so far written for it (given by \eqref{eq:decomposition-of-measure-of-curlyext} when $\Pi(S_1\cap u)=0$ 
and by \eqref{eq:caseBi-lower-bound} and \eqref{eq:caseBii-lower-bound}
when $\Pi(S_1\cap u)\neq 0$):
\begin{equation}
    \Pi(S_3 \cap u)\geq\min\left\{\frac{2}{3}e^{-3\nu}, \frac{2}{3}e^{-3\nu}-h, \frac{\psi_c}{4}e^{-\nu}h\right\}\Pi(u).
    \label{eq:the-minimum}
\end{equation}
The first term in \eqref{eq:the-minimum} is clearly larger than the other two.
To distinguish between the latter, we must make a choice of $h$. 

Subject to the constraint that the second term in \eqref{eq:the-minimum} be 
positive for all $n$, it is not possible to choose $h$ so that this term is 
smaller than the third term for all $n$.\footnote{For any \emph{fixed}
$n^{\prime}$, if we make the assignment $h=\iota e^{-3\nu}$ where $\iota$ is a constant
strictly less than $2/3$, we can choose $\iota$ to be such that the second term in \eqref{eq:the-minimum} is smaller than the third term for all $n\leq n^{\prime}$.
}
It is therefore sufficient to take $h$ to be a constant strictly smaller than $\frac{2}{3}e^{-3\nu}$. 
Since $e^{-\nu}\geq 5/6$ (see the discussion surrounding \eqref{eq:defn-of-nu}), 
any $h<\frac{2}{3}\left(\frac{5}{6}\right)^3=\frac{125}{324}$ will suffice. We will take $h=1/3$.
Thus we arrive at the bound
\begin{equation}
    \Pi(S_3 \cap u)\geq\frac{\psi_c}{12}e^{-\nu}\Pi(u).
    \label{eq:inter-bound-on-pi-S3-in-u}
\end{equation}

To write this bound in terms of $\Pi_{n-1}(\beta)$, we apply Fact \ref{fact:area-approx} to $\Pi(u)$ with $\omega=\facet$. Then, 
\begin{equation}
    \Pi(S_3 \cap u)
    \geq\frac{\psi_c}{12}\cdot (2e^{-\nu}-1)\cdot\delta\cdot\Pi_{n-1}(\beta)
    \geq \frac{\psi_c}{12}\cdot \frac{2}{3}\cdot \delta\cdot\Pi_{n-1}(\beta).
    \label{eq:lower-bound-per-facet-without-normalizing}
\end{equation}

The cube $u$ has $2n$ facets, 
which means that there are up to $2n$ facets in $\mathcal{B}$ that can contribute 
to $\Pi(S_3\cap u)$.
We will shortly sum \eqref{eq:lower-bound-per-facet-without-normalizing}
over all $\facet\in\mathcal{B}$. To avoid overcounting contributions
to $\Pi(S_3)$ when we do so, we normalize the right-hand side of
\eqref{eq:lower-bound-per-facet-without-normalizing} by a factor of $1/2n$.
Every facet $\facet\in\mathcal{B}$ therefore makes at least the
contribution 
\begin{equation}
    \frac{1}{2n}\cdot\frac{\psi_c}{18}
    \cdot\delta\cdot\Pi_{n-1}(\facet)
    \label{eq:lower-bound-per-facet}
\end{equation}
to $\Pi(S_3)$.

Now we must correct for the possibility that not all the cubes that 
contribute facets to $\mathcal{B}$ are fully contained in $K$.
To do so, we define the following subset of $C_2$:
\begin{equation}
    C_2^{\prime}=\{c\in C_2 \,\, s.t. \,\, c\cap K^{\prime}\neq\varnothing\}.
\end{equation}
By construction, every cube in $C_2^{\prime}$ and its immediate neighbors are fully 
contained in $K$. Define the set $I=K^{\prime}\cap \mathcal{C}_2$.
$I\subseteq \mathcal{C}_2^{\prime}$. Let $\mathcal{B}^{\prime}$ be the set of facets
of cubes that intersect the internal boundary of $I$ in $K^{\prime}$, $\partial_{K^{\prime}}I$.
$\mathcal{B}^{\prime}$ is the set of facets out of which the dynamics can escape
$S_1$ and that are contributed by cubes that are fully contained in $K$.
$\mathcal{B}^{\prime}\subseteq\mathcal{B}$.
Summing \eqref{eq:lower-bound-per-facet} over all facets in $\mathcal{B}^{\prime}$,
we arrive at the bound
\begin{align}
    \Pi(S_3)&\geq \frac{1}{2n}\frac{\psi_c}{18}
    \cdot\delta
    \sum_{\beta\in\mathcal{B}^{\prime}}\Pi_{n-1}(\beta)
    =\frac{1}{2n}\frac{\psi_c}{18}
    \cdot\delta\cdot
    \Pi\left(\partial_{K^{\prime}}I\right).
    \label{eq:lower-bound-on-pi-s3-before-iso}
\end{align}

$\Pi\left(\partial_{K^{\prime}}I\right)$ is lower bounded via Theorem \ref{thm:log-concave-isoperimetry} as follows:
\begin{equation}
    \Pi\left(\partial_{K^{\prime}}I\right)
    \geq\psi_{\pi}
    \min\left\{\Pi(I),\Pi(K^{\prime}{\setminus} I)\right\}.
    \label{eq:isoperimtry-on-I}
\end{equation}
We evaluate the minimum of $\Pi(I)$ and $\Pi(K^{\prime}{\setminus} I)$.
We have
\begin{equation}
    \Pi(I)
    = \Pi\left(K^{\prime}\cap \bulksetbody\right)
    > \Pi(\bulksetbody) - \varepsilon
    > \frac{1}{2}\Pi(S_1) - \varepsilon.
\end{equation}
The first inequality follows from \eqref{eq:x-in-kprime-inequality}, and the second by construction:
$\Pi(S_1)=\Pi(\bulksetbody\cap S_1)+\Pi(\bndrysetbody\cap S_1)\leq 2\Pi(\bulksetbody\cap S_1)\leq 2\Pi(\bulksetbody)$.
In addition, we have the lower bound
\begin{equation}
    \Pi(I)
    \leq \Pi(\bulksetbody^{\prime})
    < \frac{3}{2}e^{\nu}\Pi(\bulksetbody^{\prime}\cap S_1)
    \leq \frac{9}{5}\Pi(\bulksetbody^{\prime}\cap S_1)
    \leq \frac{9}{10}\Pi(K),
    \label{eq:inter-min-isoperimetry}
\end{equation}
where the second inequality follows from the definition of $C_2$, the third from \eqref{eq:defn-of-nu},
and the last from $\Pi(S_1)\leq\Pi(K)/2$.
Using \eqref{eq:inter-min-isoperimetry}, we have the following lower bound for $\Pi(K^{\prime}{\setminus}I)$:
\begin{equation}
    \Pi(K^{\prime}{\setminus}I)
    =\Pi(K^{\prime})-\Pi(I)
    \geq \Pi(K^{\prime})- \frac{9}{10}\Pi(K)
    =\frac{1}{10}\Pi(K) -\Pi(K{\setminus}K^{\prime})
    \geq\frac{1}{5}\Pi(S_1) -\varepsilon.
\end{equation}
And thus
\begin{equation}
    \min\left\{\Pi(I),\Pi(K^{\prime}{\setminus} I)\right\}
    \geq\frac{1}{5}\Pi(S_1)-\varepsilon.
    \label{eq:min-isoperimetric-thm}
\end{equation}

Putting together \eqref{eq:lower-bound-on-pi-s3-before-iso}, \eqref{eq:isoperimtry-on-I}, and \eqref{eq:min-isoperimetric-thm}, and
using the definition of $\delta$, we finally
arrive at the following bound on $\Pi(S_3)$:
\begin{equation}
    \Pi(S_3)\geq
    \frac{1}{2n}\frac{\psi_c}{18}
    \psi_{\pi}
    \frac{1}{8\sqrt{\kappa L} n\log n \max\left\{1,\sqrt{\frac{1}{n}\log\frac{1}{\varepsilon}}\right\}}
    \left(\frac{1}{5}\Pi(S_1)-\varepsilon\right).
    \label{eq:case2-lower-bound}
\end{equation}

This concludes the discussion of Case 2.

The minimum of \eqref{eq:case1-lower-bound} and \eqref{eq:case2-lower-bound} is clearly the latter.
Applying the lower bounds \eqref{eq:cube-isoperimetric coefficient} and \eqref{eq:lee-vempala-logconcave-isoperimetry} on $\psi_c$ and $\psi_{\pi}$, respectively, we finally arrive at the bound
\begin{equation*}
    \Pi(S_3)\geq
    \Psi
    \left(\frac{1}{5}\Pi(S_1)-\varepsilon\right),
\end{equation*}
where, for a universal constant $c$,
\begin{equation*}
    \Psi\geq \frac{\log 2}{2^5\cdot 3^3\cdot c}
    \frac{1}{\kappa n^{2+\frac{3}{4}}\log n \max\left\{1,\sqrt{\frac{1}{n}\log\frac{1}{\varepsilon}}\right\}}.
\end{equation*}

\end{proof}

\section{Discussion}
\label{sec:discussion}

For very large $n$, the exponent of $n$ in \eqref{eq:mixing-time-with-lee-vempala-isoperimetry} can
be improved by applying an alternative lower bound on the log-concave 
isoperimetric coefficient in Theorem \ref{thm:log-concave-isoperimetry}
due to Chen \cite{chen2021-almost-constant-lower-bound}:
\begin{equation}
    \psi_{\pi}\geq 
    \frac{\mu^{\frac{1}{2}}}{
    n^{c^{\prime}\sqrt{\frac{\log\log n}{\log n}}
    }}.
    \label{eq:chen-logconcave-isoperimetry}
\end{equation}
$c^{\prime}$ is a universal constant.
Applying this bound on $\psi_{\pi}$ in $\Psi$
instead of \eqref{eq:lee-vempala-logconcave-isoperimetry}, we have
    \begin{equation}
    \tau(\gamma)<
        C
        \kappa^2 
        n^{7+2c^{\prime}\sqrt{\frac{\log\log n}{\log n}}}
        \log^2 n 
        \left(
        \max\left\{1,\sqrt{\frac{1}{n}\log\frac{2M}{\gamma}}
        \right\}
        \right)^2
        \log\frac{2M}{\gamma}
        \label{eq:mixing-time-with-chen-isoperimetry}
\end{equation}
for universal constants $C$ and $c^{\prime}$.
The function $\sqrt{\frac{\log\log n}{\log n}}$ goes to zero as $n$ goes to infinity. 
When $n$ is large enough that $2c^{\prime}\sqrt{\frac{\log\log n}{\log n}}<0.5$, \eqref{eq:mixing-time-with-chen-isoperimetry} is a tighter bound on the mixing time than \eqref{eq:mixing-time-with-lee-vempala-isoperimetry}.

It is sometimes of interest to consider sampling from an
\emph{isotropic} log-concave distribution, i.e., a distribution for which $\kappa=1$ and $x^{\star}=0$.
Any strongly log-concave distribution can be brought into isotropic
position in $O^{\star}\left(n^5\right)$ steps \cite{vempala-lovasz-2007}.
After this preconditioning step, which need be done only once, samples may be
drawn using a sampling algorithm of choice.
If we were to consider only the subclass of isotropic log-concave distributions 
in the proof of Theorem \ref{thm:main-mixing-time}, 
we could apply concentration results for isotropic log-concave measures
that have slightly better $n$-dependence than does 
Lemma \ref{lemma:dwivedi-concentration}, such as Theorem 16 in \cite{lee-vempala-2024}, when choosing the radius of $K$. 
Doing so enables us to improve the $n$-dependence of $\delta$ to $\left(\sqrt{n}\log n\sqrt{\log\frac{1}{\varepsilon}}\max\left\{n^{1/4},\sqrt{\log\frac{1}{\varepsilon}}\right\}\right)^{-1}$ for $\varepsilon\in(0,1/2)$. This leads to a mixing time bound with a better exponent in $n$ but worse dependence on $M/\gamma$:
\begin{equation}
    \tau(\gamma)<
        C
        n^{6.5}
        \log^2 n 
        \left(
        \max\left\{n^{\frac{1}{4}},\sqrt{\frac{1}{n}\log\frac{2M}{\gamma}}
        \right\}
        \right)^2
        \log^2\frac{2M}{\gamma},
        \label{eq:mixing-time-isotropic}
\end{equation}
where $C$ is a universal constant.

We suspect that neither the bound in Theorem \ref{thm:main-mixing-time} nor \eqref{eq:mixing-time-with-chen-isoperimetry} is tight.
Establishing an isoperimetric inequality with better $n$-dependence than Lemma \ref{thm-l0-isoperimetry} is therefore an 
important problem for the future. 
A different proof technique than the one used to prove Lemma \ref{thm-l0-isoperimetry} will likely be needed.

\section*{Acknowledgements}
NSW is grateful to Yuansi Chen for suggesting the 
possibility of exploiting the approximate local uniformity of smooth functions 
to prove Lemma \ref{thm-l0-isoperimetry}, and to
Andre Wibisono, Santosh Vempala, Aditi Laddha, Samuel Livingstone, and Andrew Wade for useful discussions.
NSW would like to thank the Isaac Newton Institute for Mathematical Sciences for support and hospitality during the program Stochastic Systems for Anomalous Diffusion where a part of this work was completed. This work was supported by: EPSRC grant number EP/R014604/1.
The Flatiron Institute is a division of the Simons Foundation.

\bibliographystyle{unsrt}
\bibliography{bibliography}

\end{document}